\documentclass[11pt]{amsart}
\usepackage{amsthm,amssymb}
\usepackage{mathtools}
\usepackage[colorlinks]{hyperref}
\usepackage[margin=1in]{geometry}
\usepackage{tikz}
\usepackage{multirow}
\usepackage{float}
\usepackage{booktabs}
\usepackage{kotex}

\usepackage{caption}
\usepackage{subcaption}
\usepackage{tikz-3dplot}
\usepackage{adjustbox}

\usetikzlibrary{decorations.pathmorphing, arrows.meta}

\usepackage{tikz-cd}
\usepackage{tikz}
\usetikzlibrary{snakes, 
	3d, matrix,decorations.pathreplacing,calc,decorations.pathmorphing, patterns, automata, arrows.meta}
\usetikzlibrary {positioning}
\usetikzlibrary{matrix, fit}
\usetikzlibrary{backgrounds}

\renewcommand{\hat}{\widehat}

\newcommand{\C}{\mathbb{C}}

\newcommand{\Q}{\mathbb{Q}}
\newcommand{\R}{\mathbb{R}}

\newcommand{\Z}{\mathbb{Z}}
\newcommand{\p}{\mathbb{P}}

\newcommand{\vs}{\vspace}
\newcommand{\hs}{\hspace}

\def\mathscr{\mathscr}

\def\mcal{\mathcal}

\def\block(#1,#2)#3{\multicolumn{#2}{c}{\multirow{#1}{*}{$ #3 $}}}

\newtheorem{theorem}{Theorem}[section]

\newtheorem{lemma}[theorem]{Lemma}
\newtheorem{proposition}[theorem]{Proposition}

\newtheorem{corollary}[theorem]{Corollary}

\theoremstyle{definition}
\newtheorem{example}[theorem]{Example}
\newtheorem{definition}[theorem]{Definition}
\newtheorem{remark}[theorem]{Remark}

\numberwithin{equation}{section}

\newcommand{\red}{\textcolor{red}}
\newcommand{\blue}{\textcolor{blue}}

\newcommand{\mult}{\mathrm{mult}}
\newcommand{\Vol}{\mathrm{Vol}}

\begin{document}
	\title{Mutation of Fano Simplices and Markov type equations}
	
\begin{abstract}

It is well known that there is a bijective correspondence between the set of positive integer solutions to the Markov equation and the set of Fano triangles mutation equivalent to the Fano triangle of $\mathbb{P}^2$. In this paper, we establish a higher dimensional generalization of this correspondence for arbitrary Fano simplices of any dimension.

On the polyhedral side, we introduce a distinguished class of facets, called admissible facets, and show that their number is preserved under facet mutation. As a consequence, facet mutation classes of Fano simplices carry natural exchange graph structures whose valency is equal to the number of admissible facets. On the arithmetic side, we associate to each Fano simplex a weighted Markov-type equation together with a distinguished positive integer solution, and show that the corresponding arithmetic mutations, given by Vieta involutions, are compatible with facet mutations. More precisely, the assignment from Fano simplices to Diophantine data intertwines combinatorial mutations with arithmetic mutations, thereby relating the mutation dynamics of Fano simplices to the arithmetic dynamics of positive integer solutions. 

Finally, we introduce a piecewise linear transformation on dual polytopes, called a sliding operator, which realizes combinatorial mutation in the dual picture. As applications, we obtain a volume formula for dual simplices in terms of the associated Diophantine data and recover the multiplicity change formula under mutation.

%
%
\end{abstract}

\author{Yunhyung Cho}
\address{Department of Mathematics Education, Sungkyunkwan University, Seoul, Republic of Korea}
\email{yunhyung@skku.edu}

\maketitle

\section{Introduction}
\label{secIntroduction}

Combinatorial mutations of lattice polytopes introduced by Akhtar-Coates-Galkin-Kasprzyk \cite{ACGK} play a fundamental role in mirror symmetry for Fano varieties. They describe wall-crossing phenomena of Laurent polynomials and provide a combinatorial framework for understanding deformations of toric Fano varieties. In dimension two, combinatorial mutations exhibit deep connections with the classical Markov equation and the mutations of weighted projective planes.
A particularly striking manifestation of this phenomenon appears in the work of Hacking--Prokhorov~\cite{HP} which establishes a correspondence between weighted projective planes admitting $\Q$-Gorenstein smoothings to $\mathbb{P}^2$ and positive integer solutions of the Markov equation
\[
	x^2+y^2+z^2 = 3xyz.
\]
The mutation dynamics of Markov triples are governed by Vieta involutions producing an infinite exchange graph structure on the set of positive integer solutions. 
This interaction between polyhedral geometry, Diophantine equations, and mutation dynamics has subsequently appeared in several contexts related to mirror symmetry, cluster algebras, and toric degenerations. See \cite{HP}, \cite{Vi1}, \cite{Vi2}, \cite{GR} for example.

\begin{figure}[h]
\centering

\begin{tikzpicture}[
    scale=1.0,
    every node/.style={font=\small},
    vertex/.style={circle, fill=black, inner sep=1.6pt},
    edge/.style={line width=0.8pt}
]


\node[vertex,label={[xshift=0mm,yshift=-3mm]below:{$(1,1,1)$}}] (O) at (0,0) {};


\node[vertex,label={[xshift=7mm,yshift=0mm]below:{$(1,1,2)$}}] (U1) at (0,3) {};
\draw[edge] (O)--(U1);

\node[vertex,label={[xshift=7mm,yshift=-3mm]:{$(1,5,2)$}}] (UL) at (-2,4) {};
\node[vertex,label={[xshift=-7mm,yshift=-3mm]:{$(5,1,2)$}}] (UR) at (2,4) {};

\draw[edge] (U1)--(UL);
\draw[edge] (U1)--(UR);

\node[vertex,label=above:{$(1,5,13)$}] (ULU) at (-2.7,5.4) {};
\node[vertex,label=left:{$(2,5,29)$}] (ULD) at (-3.5,3.1) {};

\draw[edge] (UL)--(ULU);
\draw[edge] (UL)--(ULD);

\node[vertex,label=above:{$(5,29,2)$}] (URU) at (2.8,5.4) {};
\node[vertex,label=right:{$(5,1,13)$}] (URD) at (3.5,3.1) {};

\draw[edge] (UR)--(URU);
\draw[edge] (UR)--(URD);


\node[vertex,label={[xshift=0mm,yshift=-1mm]right:{$(2,1,1)$}}] (L1) at (-3,-2) {};
\draw[edge] (O)--(L1);

\node[vertex,label={[xshift=6mm,yshift=-2mm]above:{$(2,5,1)$}}] (LUL) at (-5,-1) {};
\node[vertex,label={[xshift=-6mm,yshift=-2mm]above:{$(2,1,5)$}}] (LLL) at (-4,-4) {};

\draw[edge] (L1)--(LUL);
\draw[edge] (L1)--(LLL);

\node[vertex,label=left:{$(13,5,1)$}] (LUL1) at (-5.5,0.5) {};
\node[vertex,label=left:{$(2,5,29)$}] (LUL3) at (-6.2,-2.0) {};

\draw[edge] (LUL)--(LUL1);
\draw[edge] (LUL)--(LUL3);

\node[vertex,label=left:{$(2,29,5)$}] (LLL1) at (-5.5,-4.5) {};
\node[vertex,label=left:{$(13,1,5)$}] (LLL2) at (-3,-5.4) {};

\draw[edge] (LLL)--(LLL1);
\draw[edge] (LLL)--(LLL2);


\node[vertex,label={[xshift=0mm,yshift=-1mm]left:{$(1,2,1)$}}] (R1) at (3,-2) {};
\draw[edge] (O)--(R1);

\node[vertex,label={[xshift=-6mm,yshift=-2mm]above:{$(1,2,5)$}}] (RUR) at (5,-1) {};
\node[vertex,label={[xshift=6mm,yshift=-2mm]above:{$(5,2,1)$}}] (RLR) at (4,-4) {};

\draw[edge] (R1)--(RUR);
\draw[edge] (R1)--(RLR);

\node[vertex,label=right:{$(29,2,5)$}] (RUR1) at (5.5,0.5) {};
\node[vertex,label=right:{$(1,13,5)$}] (RUR3) at (6.2,-2.0) {};

\draw[edge] (RUR)--(RUR1);
\draw[edge] (RUR)--(RUR3);

\node[vertex,label=right:{$(5,13,1)$}] (RLR1) at (5.5,-4.5) {};
\node[vertex,label=right:{$(5,2,29)$}] (RLR2) at (3,-5.4) {};

\draw[edge] (RLR)--(RLR1);
\draw[edge] (RLR)--(RLR2);

\end{tikzpicture}

\caption{Exchange graph for $x^2 + y^2 + z^2 = 3xyz$.}

\end{figure}
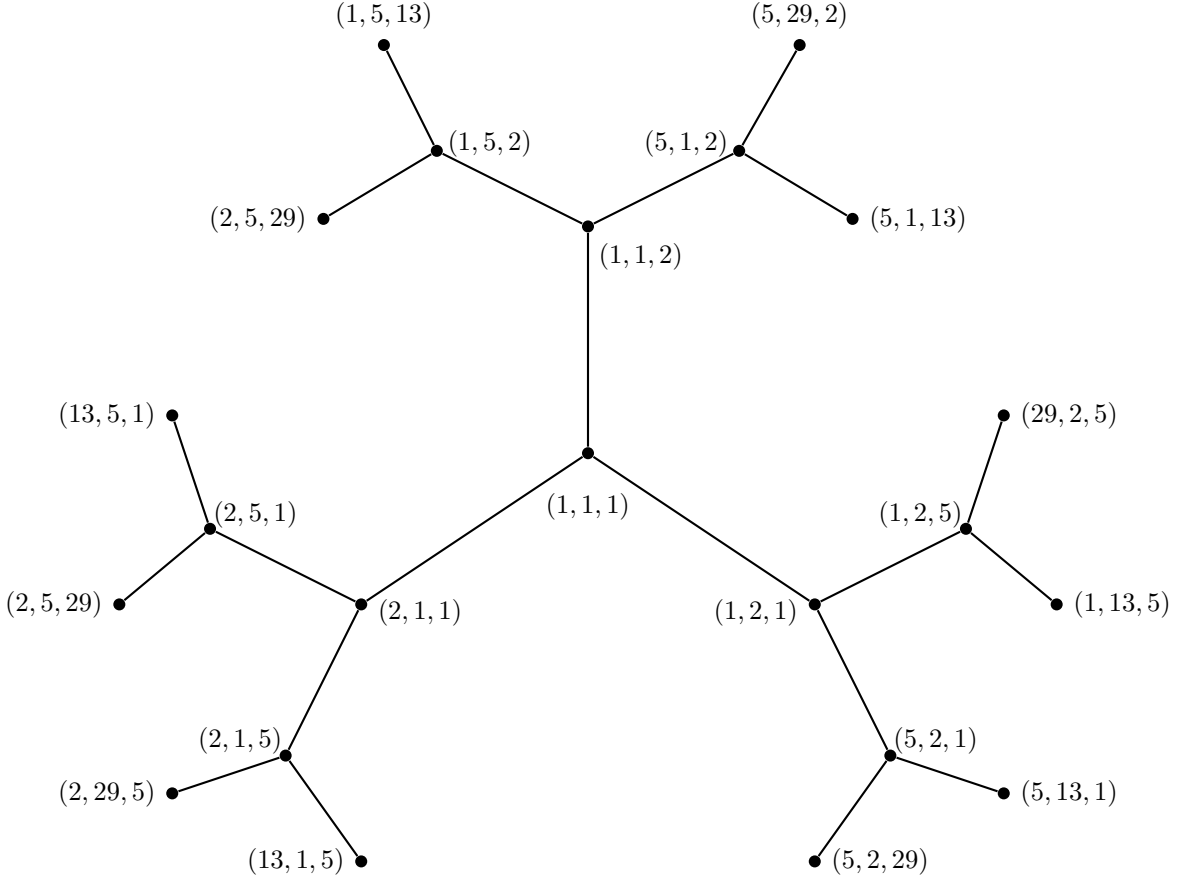

The purpose of this paper is to show that this picture extends naturally to higher-dimensional Fano simplices. More precisely, we develop   a mutation-theoretic framework for fake weighted projective spaces in which combinatorial mutations of simplices over admissible facets are intertwined with Vieta-type mutations on Markov tuples of weighted Markov-type equations: 
\[
	k(c_0x_0^2+\cdots+c_nx_n^2)=mx_0x_1\cdots x_n.
\]
Our starting point is the introduction of a distinguished class of facets of a Fano simplex, called \emph{admissible facets}. Roughly speaking, a facet is admissible if it is divisible by its affine distance from the origin. For each admissible facet, one can define a canonical combinatorial mutation over the facet. We prove that the number of admissible facets is preserved under such mutations.
As a consequence, each mutation equivalence class of Fano simplices admits a natural exchange graph structure whose valency is determined by the number of admissible facets. This provides a higher-dimensional combinatorial analogue of the exchange structures arising from mutations of Markov triples.

On the arithmetic side, we associate to each Fano simplex a weighted Markov-type equation
\[
	k(c_0x_0^2+\cdots+c_nx_n^2)=mx_0x_1\cdots x_n
\]
together with a distinguished positive integer solution. The associated arithmetic mutations are given by higher-dimensional Vieta involutions. Our main result shows that these arithmetic mutations are compatible with combinatorial mutations of Fano simplices.
More precisely, we prove that the assignment
\[
P \longmapsto (DE_P,a_P)
\]
from Fano simplices to Diophantine data intertwines geometric mutations with arithmetic mutations. In this way, we obtain a higher-dimensional extension of the classical correspondence between Markov triples and weighted projective planes.

Our main theorem may be summarized as follows.

\begin{theorem}
Let $P$ be a Fano simplex. Then
\begin{enumerate}
\item The number of admissible facets of $P$ is invariant under facet mutation. In particular, the mutation equivalence class of $P$ carries a canonical exchange graph whose edges correspond to admissible facet mutations.

\item There exists a weighted Markov-type equation $DE_P$ associated with $P$ that is invariant under facet mutation. Moreover, the set of positive integer solutions of $DE_P$ carries a canonical exchange graph whose edges correspond to Vieta involutions.
\item One can also associate to each Fano simplex $P$ a distinguished positive integer solution $a_P$ of $DE_P$. Moreover, the correspondence
\[
	P \longmapsto (DE_P,a_P)
\]
is compatible with mutations.

\end{enumerate}
\end{theorem}

This reveals a bridge between mutation dynamics of Fano simplices and arithmetic mutation dynamics of Markov-type equations. From this perspective, facet-mutation classes of Fano simplices may be viewed as geometric counterparts of mutation classes of Diophantine equations.
We note that, while mutation-invariant Diophantine equations are known in dimension two \cite{Ak, AK1}, 
they do not persist for general mutations in higher dimensions \cite[Remark~17]{CGKN}.
One of the main observations of this paper is that such invariance reappears when one restricts to admissible-facet mutations.

We further develop a dual polyhedral interpretation of combinatorial mutation. More precisely, we introduce a piecewise linear transformation on dual polytopes, called a \emph{sliding operator}, which realizes combinatorial mutation in the dual picture. This operator preserves volume and provides a geometric interpretation of the associated arithmetic data. As applications, we derive an explicit formula for the volume of the dual simplex in terms of the coefficients of the associated Markov-type equation and the multiplicity of the simplex. We also recover the multiplicity change formula under mutation obtained in~\cite{CGKN}.

The paper is organized as follows. In Section~2, we review combinatorial mutations of lattice polytopes. In Section~3, we introduce admissible facets and describe the exchange graph structure associated to a facet-mutation class of Fano simplices. In Section~4, we associate weighted Markov-type equations to Fano simplices and study their arithmetic mutation dynamics. In Section~5, we develop the dual viewpoint via sliding operators and derive volume formulas for dual simplices.

\section{Combinatorial mutations}
\label{secCombinatorialMutations}
	
	In this section, we review the notion of combinatorial mutation introduced in \cite{ACGK} and its properties.
	Let $N$ be a lattice of rank $n$ with dual lattice $M$ and denote by $N_\R := N \otimes_\Z \R$
	and $M_\R := M \otimes_\Z \R$, respectively. 
	Let $P$ be a lattice polytope in $N_\R$ containing the origin ${\bf 0}$ in its interior. For a given 
	primitive vector ${\bf w} \in M$, we use the following notation:
	\begin{itemize}
		\item $h_{\min}(P,{\bf w}) := \min ~\{ \langle {\bf w}, v \rangle ~|~ v \in P \}$, 
		the {\em minimal ${\bf w}$-height}, 
		\item $h_{\max}(P,{\bf w}) := \max ~\{ \langle {\bf w}, v \rangle ~|~ v \in P \}$, 
		the {\em maximal ${\bf w}$-height}, 
		\item $H_h(P,{\bf w}) := \{ x \in N_\R ~|~ \langle {\bf w}, x \rangle = h \}$,
		the hyperplane of ${\bf w}$-height $h$, and
		\item $S_{h}(P,{\bf w}) := \text{conv}(H_h(P,{\bf w}) \cap P \cap N)$,
		the maximal lattice polytope in the slice polytope $H_h(P,{\bf w}) \cap P$ at 
		${\bf w}$-height $h$. (Note that $S_{h}(P,{\bf w})$ is possibly empty if there is no lattice point in the slice polytope.)

	\end{itemize}
	
	Now we consider a lattice polytope $G \subset {\bf w}^\perp$ such that \vs{0.1cm}
	
\noindent
\begin{tikzpicture}
\node[anchor=north west] at (-1.2,0) {$\left(\ast\right)$};
\node[anchor=north west, text width=0.9\textwidth, align=justify] at (-0.5,0) {
for each integer $h$ with $h_{\min}(P,{\bf w}) \leq h < 0$, there exists a (possibly empty) lattice polytope $L_h$ satisfying 
\[
	H_{h}(P,{\bf w}) \cap \mcal{V}(P) \subseteq L_h + (-h)G \subseteq S_h(P,{\bf w})
\]
where $\mcal{V}(P)$ denotes the set of vertices of $P$.
(For example, if $H_{h}(P,{\bf w}) \cap \mcal{V}(P) = \emptyset$, then $L_h$ can be taken to be an empty set.)
};
\end{tikzpicture}

\noindent
\begin{definition}\label{def_comb_mutation}\cite[Definition 5]{ACGK}
	A {\em combinatorial mutation} of $P$ along $({\bf w}, G, \{L_h\})$ is defined to be the lattice
	polytope 
	\[
		\text{mut}_{\bf w}(P,G;\{L_h\}) := \text{conv} \left( \bigcup_{h=h_{\min}}^{-1} L_h \cup
		\bigcup_{h=0}^{h_{\max}} (S_{h}(P,{\bf w}) + hG) \right) \subset N_\R.
	\]
	We call ${\bf w}$ a {\em height vector} and $G$ a {\em factor} for the mutation.
\end{definition}

Note that the choice of the family $\{L_h\}$ satisfying $\left(\ast\right)$ may not be unique. 
It turned out that, up to unimodular equivalence, the combinatorial mutation $\text{mut}_{\bf w}(P,G;\{L_h\})$ does not depend on the choice of the family $\{ L_h \}$, and therefore we may write 
\[
	\text{mut}_{\bf h}(P,G) := \text{mut}_{\bf w}(P,G;\{L_h\}). 
\]
See \cite[Proposition 1]{ACGK}. Moreover, the combinatorial mutation is involutive in the sense that 
\[
	\text{mut}_{-{\bf w}}(\text{mut}_{\bf w}(P,G), G) \cong P
\]
where `$\cong$' denotes the unimodular equivalence.

\begin{example} 
\label{example_standard_simplex}

Let $P = \text{conv}({\bf e}_1,{\bf e}_2, -{\bf e}_1 - {\bf e}_2)$ be the standard $2$-simplex in $\R^2$. Let ${\bf w} = (-1,-1)$ and a factor $G = \text{conv}({\bf 0}, {\bf e}_1 - {\bf e}_2)$. Then 
$L_{-1} = \{ {\bf e}_2 \}$ satisfies the condition $\left(\ast\right)$ and the combinatorial mutation
$\text{mut}_{\bf w}(P,G,\{L_{-1}\})$ is given by the red triangle below.

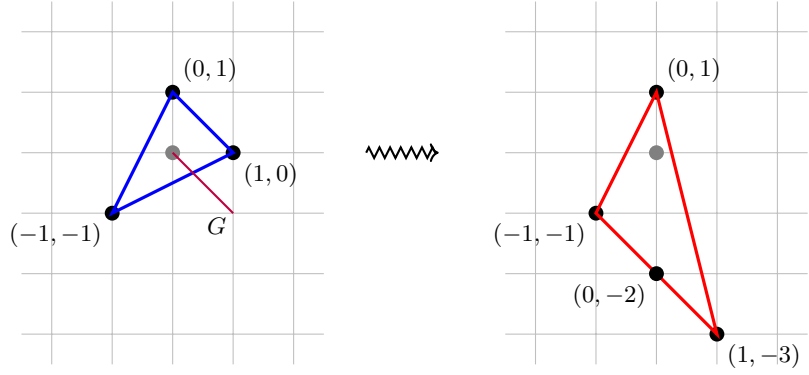
\begin{center}
\begin{tikzpicture}[scale=0.8]

  \begin{scope}
    \draw[step=1cm, gray!50, thin] (-2.5,-3.5) grid (2.5,2.5);

    \fill[gray] (0,0) circle (3.5pt);
    \fill[black] (0,1) circle (3.5pt);
    \fill[black] (1,0) circle (3.5pt);
    \fill[black] (-1,-1) circle (3.5pt);

    \draw[very thick, blue] (0,1) -- (1,0) -- (-1,-1) -- cycle;

 \draw[thick, purple] (0,0) -- (1,-1);

    \node[below right] at (0.4,-0.9) {\footnotesize $G$};
    \node[above right] at (0,1) {\footnotesize $(0,1)$};
    \node[below right] at (1,0) {\footnotesize $(1,0)$};
    \node[below left] at (-1,-1) {\footnotesize $(-1,-1)$};
  \end{scope}

  \draw[decorate, decoration={zigzag, segment length=4, amplitude=2}, thick, ->]
    (3.2,0) -- (4.4,0);

  \begin{scope}[xshift=8.0cm]
    \draw[step=1cm, gray!50, thin] (-2.5,-3.5) grid (2.5,2.5);
    \fill[gray] (0,0) circle (3.5pt);
    \fill[black] (0,1) circle (3.5pt);
    \fill[black] (-1,-1) circle (3.5pt);
    \fill[black] (1,-3) circle (3.5pt);
    \draw[very thick, red] (0,1) -- (-1,-1) -- (1,-3) -- cycle;
    \fill[black] (0,-2) circle (3.5pt);
    \node[above right] at (0,1) {\footnotesize $(0,1)$};
    \node[below left] at (-1,-1) {\footnotesize $(-1,-1)$};
    \node[below right] at (1,-3) {\footnotesize $(1,-3)$};
    \node[below left] at (0,-2) {\footnotesize $(0,-2)$};
  \end{scope}

\end{tikzpicture}
\end{center}

\end{example}

\begin{example} 

Let $P$ be the square with four vertices $\{ \pm {\bf e}_1 \pm {\bf e}_2 \}$. Let ${\bf w} = (0,-1)$ be a height vector and consider two possible factors $G_1 = \text{conv}({\bf 0}, {\bf e}_1)$ and $G_2 = \text{conv}({\bf 0}, 2{\bf e}_1)$. For each factor, we have $(L_1)_{-1} = \text{conv}(-{\bf e_1} + {\bf e}_2, {\bf e}_2)$ and $(L_2)_{-1} = \{ -{\bf e_1} + {\bf e}_2\}$. Then the combinatorial mutations of $P$ along $({\bf w}, G_1, \{ (L_1)_{-1} \})$ and $({\bf w}, G_2, \{ (L_2)_{-1}\})$ are described below.

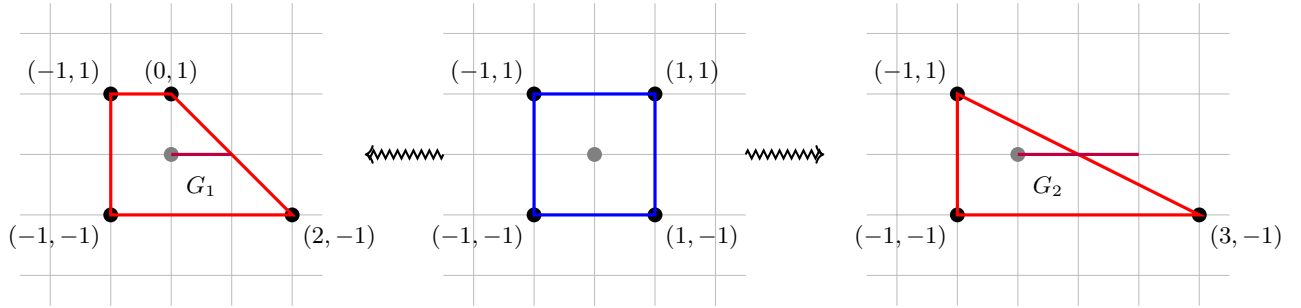
\begin{center}
\begin{tikzpicture}[scale=0.8]

  \begin{scope}[xshift=-7cm]
    \draw[step=1cm, gray!50, thin] (-2.5,-2.5) grid (2.5,2.5);
    \fill[gray] (0,0) circle (3.5pt);
    \foreach \x/\y in {-1/1, 0/1, -1/-1, 2/-1} {
      \fill[black] (\x,\y) circle (3.5pt);
    }
    \draw[very thick, red] (-1,1) -- (0,1) -- (2,-1) -- (-1,-1) -- cycle;
    \draw[very thick, purple] (0,0) -- (1,0);
    \node[below] at (0.5,-0.2) {\footnotesize $G_1$};
    \node[above left]  at (-1,1)  {\footnotesize $(-1,1)$};
    \node[above]       at (0,1)   {\footnotesize $(0,1)$};
    \node[below right] at (2,-1)  {\footnotesize $(2,-1)$};
    \node[below left]  at (-1,-1) {\footnotesize $(-1,-1)$};
  \end{scope}

  \begin{scope}
    \draw[step=1cm, gray!50, thin] (-2.5,-2.5) grid (2.5,2.5);
    \fill[gray] (0,0) circle (3.5pt);
    \foreach \x/\y in {-1/1, 1/1, 1/-1, -1/-1} {
      \fill[black] (\x,\y) circle (3.5pt);
    }
    \draw[very thick, blue] (-1,1) -- (1,1) -- (1,-1) -- (-1,-1) -- cycle;
    \node[above left]  at (-1,1)  {\footnotesize $(-1,1)$};
    \node[above right] at (1,1)   {\footnotesize $(1,1)$};
    \node[below right] at (1,-1)  {\footnotesize $(1,-1)$};
    \node[below left]  at (-1,-1) {\footnotesize $(-1,-1)$};
  \end{scope}

  \begin{scope}[xshift=7cm]
    \draw[step=1cm, gray!50, thin] (-2.5,-2.5) grid (3.5,2.5);
    \fill[gray] (0,0) circle (3.5pt);
    \foreach \x/\y in {-1/1, -1/-1, 3/-1} {
      \fill[black] (\x,\y) circle (3.5pt);
    }
    \draw[very thick, red] (-1,1) -- (-1,-1) -- (3,-1) -- cycle;
    \draw[very thick, purple] (0,0) -- (2,0);
    \node[below] at (0.5,-0.2) {\footnotesize $G_2$};
    \node[above left]  at (-1,1)  {\footnotesize $(-1,1)$};
    \node[below left]  at (-1,-1) {\footnotesize $(-1,-1)$};
    \node[below right] at (3,-1)  {\footnotesize $(3,-1)$};
  \end{scope}

  \draw[decorate, decoration={zigzag, segment length=3, amplitude=1.5}, thick, ->]
    (-2.5,0) -- (-3.8,0);

  \draw[decorate, decoration={zigzag, segment length=3, amplitude=1.5}, thick, ->]
    (2.5,0) -- (3.8,0);

\end{tikzpicture}
\end{center}
\end{example}

\noindent
We note that the combinatorial mutation is compatible with the mutation of a Laurent polynomial 
in the sense that the following diagram commutes: 
\[
	\begin{array}{cccc}
		f & \mapsto & \widetilde{f} = \text{mut}_{\bf w}(f, g) &   \\
		\downarrow  & & \downarrow &  \\
		P:= \text{Newt}(f) & \mapsto & \text{Newt}(\widetilde{f}) = \text{mut}_{\bf w}(P,G), & G:= \text{Newt}(g)
	\end{array}
\]
See \cite{GU} and \cite{ACGK} for more detail.

A lattice polytope $P \subset N_\R$ is called {\em Fano} if $P$ contains ${\bf 0}$ in its interior and each vertex of $P$ is primitive. It is well-known that Fano polytopes are in bijective correspondence with $\Q$-Gorenstein Fano (or simply Fano) toric varieties. The following is the most important property of the combinatorial mutation.

\begin{proposition}\cite[Proposition 2]{ACGK}
	The combinatorial mutation of a Fano polytope is again a Fano polytope. 
\end{proposition}	

\noindent
Remarkably, Ilten~\cite{Ilt} proved that if $Q$ is obtained from $P$ by a combinatorial mutation, then there exists a flat family of projective varieties over $\p^1$ whose fibers over $0$ and $\infty$ are the toric varieties associated with $P$ and $Q$, respectively.
In dimension two, such mutations are closely related to $\Q$-Gorenstein deformations of weighted projective planes studied by Hacking and Prokhorov \cite{HP}.

\section{Mutations of Fano simplices and admissible facets}
\label{secMutationsOfFanoSimplicesAndAdmissibleFacets}

In this section, we restrict our attention to Fano simplices and combinatorial mutations over facets. We introduce a distinguished class of facets, called admissible facets, and study the corresponding facet mutations. The significance of admissible facets will become apparent in Section \ref{secMarkovtypeEquationsOfFanoSimplices}, where they naturally give rise to arithmetic mutations of distinguished solutions of weighted Markov-type equations.

For mutations over admissible facets, the general construction of combinatorial mutation admits a particularly simple description. Let $P \subset N_{\mathbb R}$ be a Fano simplex and let $F$ be a facet of $P$ with supporting hyperplane
\[
H_F:=\{u\in N_{\mathbb R}\mid \langle n_F,u\rangle=d_F\},
\]
where $n_F$ is the primitive outward normal vector of $F$ and $d_F>0$ is the affine distance from the origin to $F$.	
\begin{definition}\label{def_admissible}
	A facet $F$ of $P$ is called {\em admissible} if $F$ is a $d_F$-fold dilation of some lattice     		polytope up to translation. 
	Equivalently, $F$ is admissible if for a (any) vertex $v_0$ of $F$, 
	there exists a lattice polytope $G_{F,v_0} \subset N_\R$ with a vertex at the origin
	such that 
	\[
		F = v_0 + d_F G_{F,v_0}.
	\]
	We denote the number of admissible facets of $P$ by $\mathrm{ad}(P)$. 
\end{definition}

\noindent
One can easily check that the {\em admissibility} is well-defined regardless of the choice of $v_0$, whereas $G_{F,v_0}$ does depend on the choice. 

If $F$ is admissible, then there is a canonical way of taking a height vector and a factor so that we may define a combinatorial mutation {\em over a facet}
as follows. Take the primitive inward normal vector ${\bf w}_F := -n_F$ as a height vector. We also fix a vertex $v_0 \in F$ and let $G := G_{F,v_0} \in {\bf w}_F^\perp$ be defined above so that 
\[
	h_{\min}(P,{\bf w}_F) = -d_F.
\]
If we take $L_{-d_F} := \{ v_0 \}$, then it satisfies the $(\ast)$ condition:
\[
	\underbrace{H_{-d_F}(P,{\bf w}_F) \cap \mcal{V}(P)}_{= \mcal{V}(F)} \hs{0.2cm} \subseteq
\hs{0.2cm} \underbrace{L_{-d_F} + d_F G}_{= F} \hs{0.2cm} \subseteq \hs{0.2cm}
	S_{-d_F}(P,{\bf w}_F) = F.
\]	
Since there is no vertex of $P$ not in $\mcal{V}(F)$ with negative ${\bf w}_F$-height as $P$ is a simplex, we may conclude that $G$ is a factor of $P$ with $L_{-d_F} = \{v_0 \}$.  

\begin{definition}
	For an admissible facet $F$ of $P$, the combinatorial mutation 
	\[
		\mu_{F,v_0}(P) := \text{mut}_{{\bf w}_F}(P,G_{F,v_0}; \{L_{-d_F} \})
	\]
	is called a {\em mutation of $P$ over the facet $F$}.
\end{definition}

\noindent
Note that if $v_1$ is another vertex of $F$, then the factors $G_{F,v_0}$ and $G_{F,v_1}$ are differ by 
\[
	G_{F,v_1} = G_{F,v_0} + (v_0 - v_1).
\]
On the other hand, one easily checks that 
$\text{mut}_{\bf w}(P,G;\{L_h\}) \cong \text{mut}_{\bf w}(P,G+v;\{L_h + hv\})$
for any $v \in N$ with ${\bf w}(v) = 0$.
In other words, the choice of the factor, up to translation, 
does not affect the combinatorial mutation up to unimodular equivalence.
(See \cite[Proposition 1]{ACGK}.)
Thus the combinatorial mutation of $P$ over $F$ is independent of the choice of a vertex of $F$ and hence we may simply denote it by $\mu_F(P):= \mu_{F,v_0}(P)$ without ambiguity.
We are now ready to state our main theorem.

\begin{theorem}\label{thm_main}
	Let $P \subset N_\R$ be an $n$-dimensional Fano simplex and $F$ an admissible facet of $P$. Then 
	\[
		\mathrm{ad}(P) = \mathrm{ad}(\mu_F(P)).
	\]
\end{theorem}

\begin{proof}
	Let $F$ be an admissible facet of $P$.
	Without loss of generality, we may assume that 
	\[
   		\mcal{V}(P) = \{ v_0, v_1, \dots, v_n \} \quad \text{and} \quad 
   		F = \mathrm{conv}\{v_0, v_1, \dots, v_{n-1}\}.
	\]
	We first fix some notation as follows. 
	\begin{itemize}
		\item $F_i := \text{conv}(v_0,v_1,\dots,\hat{v}_i,\dots,v_n)$, in particular $F_n = F$,  
   		\item ${\bf w}_n := -n_{F}$, the inward primitive normal vector of the facet $F$,
   		\item $h_{\min} := -d_F = \langle {\bf w}_n, v_0 \rangle = \dots = \langle {\bf w}_n, v_{n-1} \rangle\in \Z_{<0}$,
   		\item $h_{\max} := \langle {\bf w}_n, v_n \rangle \in \Z_{>0}$,
   		\item $L_{h_{\min}} := \{v_0\}$,
   		\item $G := \displaystyle \frac{1}{d_F} (F - v_0)$.
	\end{itemize}   
	By definition, we have that 
	\[
		\mu_F(P) := \mu_{F,v_0}(P) = 
		\text{mut}_{{\bf w}_n}(P,G; \{L_{h_{\min}} \}) = \text{conv}(v_0, v_n + h_{\max}G).
	\]	
   	We label each vertices of $\mu_F(P)$ by  
   	\[
   		v_i' := \begin{cases}
   			v_{n-i} & \text{if $i=0$ or $i=n$,} \\
   			v_n - \frac{h_{\max}}{h_{\min}} (v_i - v_0) & \text{otherwise}.
   		\end{cases}
   	\]
	See Figure \ref{figure_mutation}.

\tdplotsetmaincoords{70}{120}

\begin{figure}[htbp]
\centering
\begin{adjustbox}{center}

\begin{tikzpicture}[tdplot_main_coords, scale=2.2]

  \draw[gray!60,->] (0,0,0) -- (1.7,0,0) node[below right] {\small $$};
  \draw[gray!60,->] (0,0,0) -- (0,1.2,0) node[above left] {\small $$};
  \draw[gray!60,->] (0,0,0) -- (0,0,0.8) node[above] {\small $$};
  \filldraw[gray!70] (0,0,0) circle (0.8pt);

  \coordinate (W2) at (1.7, 0.4, 0);
  \coordinate (W3) at (-0.2, 0.4, 0);
  \coordinate (W0) at (0, -0.1, 0.4);
  \coordinate (W1) at (-0.2, -0.5, 0);

  \draw[very thick, blue] (W2) -- (W1);
  \draw[very thick, blue] (W1) -- (W0);
  \draw[very thick, blue] (W0) -- (W2);

  \draw[very thick] (W2) -- (W3);
  \draw[very thick] (W3) -- (W0);
  \draw[dashed, very thick] (W3) -- (W1);

  \node at ($(W0) + (0.1,-0.15,0)$) [above right] {\small $v_0$};
  \node at ($(W1) + (-0.1,0,0.1)$) [below left] {\small $v_1$};
  \node at ($(W2) + (0.1,-0.1,-0.1)$) [right] {\small $v_2$};
  \node at ($(W3) + (0,0.35,0)$) [left] {\small $v_3$};

\end{tikzpicture}
\hspace{2cm}

\begin{tikzpicture}[tdplot_main_coords, scale=2.2]

  \draw[gray!60,->] (0,0,0) -- (1.7,0,0) node[below right] {\small $$};
  \draw[gray!60,->] (0,0,0) -- (0,1.2,0) node[above left] {\small $$};
  \draw[gray!60,->] (0,0,0) -- (0,0,0.8) node[above] {\small $$};
  \filldraw[gray!70] (0,0,0) circle (0.8pt);

  \coordinate (W2) at (1.7, 0.4, 0);
  \coordinate (W3) at (-0.2, 0.4, 0);
  \coordinate (W0) at (0, -0.1, 0.4);
  \coordinate (W1) at (-0.2, -0.5, 0);

  \coordinate (U1) at (-0.4, -0.0, -0.4);
  \coordinate (U2) at (1.5, 0.9, -0.4);

  \draw[dashed, semithick, blue] (W2) -- (W1);
  \draw[dashed, semithick, blue] (W1) -- (W0);
  \draw[dashed, semithick, blue] (W0) -- (W2);

  \draw[very thick] (W3) -- (W0);

  \draw[very thick] (W0) -- (U2);
  \draw[dashed, very thick, red] (U1) -- (W3);
  \draw[dashed, very thick, red] (U1) -- (U2);
  \draw[very thick, red] (U2) -- (W3);

  \draw[dashed, very thick] (W0) -- (U1);

  \node at ($(W0) + (0.1,-0.15,0)$) [above right] {\small $v_0 = v_3'$};
  \node at ($(W1) + (-0.1,0,0.1)$) [below left] {\small $v_1$};
  \node at ($(W2) + (0.1,-0.1,-0.1)$) [right] {\small $v_2$};
  \node at ($(W3) + (0,-0.05,0.05)$) [right] {\small $v_3 = v_0'$};

  \node at ($(U1) + (0,-0.33,-0.05)$) [right] {\small $v_1'$};
  \node at ($(U2) + (-0.05,-0.03,0)$) [left] {\small $v_2'$};

\end{tikzpicture}
\end{adjustbox}

\caption{Combinatorial mutation of a Fano simplex; $F_n$ and $F_n'$}
\label{figure_mutation}
\end{figure}
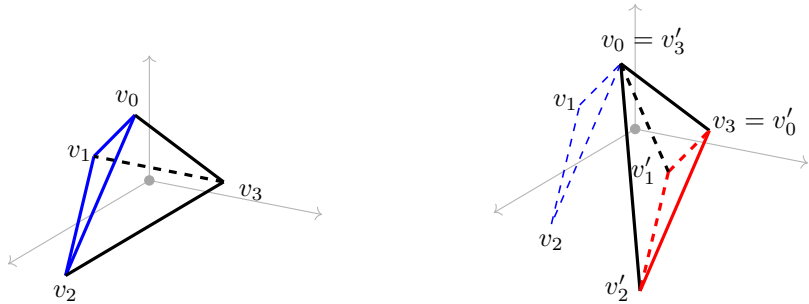

\noindent
We will show that if $F_i$ is admissible in $P$, then 
$F_i' := \text{conv} (v_0',v_1',\dots,\hat{v}_i',\dots,v_n')$ is also admissible in 
$\mu_F(P)$. This shows that $\mathrm{ad}(P) \leq \mathrm{ad}(\mu_F(P))$ and
the proof follows from the involutivity of the combinatorial mutation.

\noindent
Denote by $d_i \in \Z_{>0}$ the affine distance from $F_i$ to the origin (and similarly by $d_i'$ for $F_i'$) and define 
\[
	\ell_{i,j} := \mathrm{lcm}(d_0, \dots, \hat{d}_i, \dots, \hat{d}_j, \dots, d_n).
\] 
We further denote by $n_{F_i}$ the outward primitive integral vector normal to $F_i$ and set ${\bf w}_i := -n_{F_i}$ (and similarly by ${\bf w}_i' = -n_{F_i'}$). 
\vs{0.1cm}

\noindent
{\bf Case I: $i = n$.} 
\vs{0.1cm}

\noindent
Note that $F_n = F$ is admissible by assumption. Since ${\bf w}_n' = -n_{F_n'} = n_{F_n}$, we have that
\[
	h_{\min}' := \langle {\bf w}_n', F_n' \rangle = \langle n_{F_n}, v_0' \rangle = \langle n_{F_n}, v_n 	\rangle = -h_{\max}. 
\]
Since $d_n' = h_{\max}$, it follows that $F_n' - v_0' = h_{\max}G = d_n'G$ is divisible by $d_n'$, and therefore $F_n'$ is admissible. See Figure \ref{figure_mutation}; the blue and red triangles indicate $F_n$ and $F_n'$, respectively.
\vs{0.1cm}

\noindent
{\bf Case II: $i \in \{1,\dots,n-1\}$.} 
\vs{0.1cm}

\tdplotsetmaincoords{70}{120}

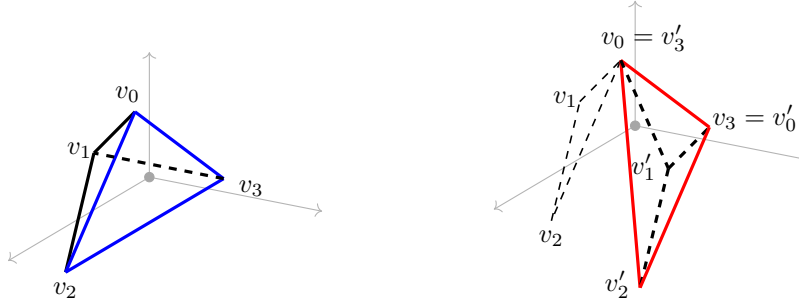
\begin{figure}[htbp]
\centering
\begin{adjustbox}{center}

\begin{tikzpicture}[tdplot_main_coords, scale=2.2]

  \draw[gray!60,->] (0,0,0) -- (1.7,0,0) node[below right] {\small $$};
  \draw[gray!60,->] (0,0,0) -- (0,1.2,0) node[above left] {\small $$};
  \draw[gray!60,->] (0,0,0) -- (0,0,0.8) node[above] {\small $$};
  \filldraw[gray!70] (0,0,0) circle (0.8pt);

  \coordinate (W2) at (1.7, 0.4, 0);
  \coordinate (W3) at (-0.2, 0.4, 0);
  \coordinate (W0) at (0, -0.1, 0.4);
  \coordinate (W1) at (-0.2, -0.5, 0);

  \draw[very thick] (W2) -- (W1);
  \draw[very thick] (W1) -- (W0);
  \draw[very thick, blue] (W0) -- (W2);

  \draw[very thick, blue] (W2) -- (W3);
  \draw[very thick, blue] (W3) -- (W0);
  \draw[dashed, very thick] (W3) -- (W1);

  \node at ($(W0) + (0.1,-0.15,0)$) [above right] {\small $v_0$};
  \node at ($(W1) + (-0.1,0,0.1)$) [below left] {\small $v_1$};
  \node at ($(W2) + (0.1,-0.1,-0.1)$) [right] {\small $v_2$};
  \node at ($(W3) + (0,0.35,0)$) [left] {\small $v_3$};

\end{tikzpicture}
\hspace{2cm}

\begin{tikzpicture}[tdplot_main_coords, scale=2.2]

  \draw[gray!60,->] (0,0,0) -- (1.7,0,0) node[below right] {\small $$};
  \draw[gray!60,->] (0,0,0) -- (0,1.2,0) node[above left] {\small $$};
  \draw[gray!60,->] (0,0,0) -- (0,0,0.8) node[above] {\small $$};
  \filldraw[gray!70] (0,0,0) circle (0.8pt);

  \coordinate (W2) at (1.7, 0.4, 0);
  \coordinate (W3) at (-0.2, 0.4, 0);
  \coordinate (W0) at (0, -0.1, 0.4);
  \coordinate (W1) at (-0.2, -0.5, 0);

  \coordinate (U1) at (-0.4, -0.0, -0.4);
  \coordinate (U2) at (1.5, 0.9, -0.4);

  \draw[dashed, semithick] (W2) -- (W1);
  \draw[dashed, semithick] (W1) -- (W0);
  \draw[dashed, semithick] (W0) -- (W2);

  \draw[very thick,red] (W3) -- (W0);

  \draw[very thick, red] (W0) -- (U2);
  \draw[dashed, very thick] (U1) -- (W3);
  \draw[dashed, very thick] (U1) -- (U2);
  \draw[very thick, red] (U2) -- (W3);

  \draw[dashed, very thick] (W0) -- (U1);

  \node at ($(W0) + (0.1,-0.15,0)$) [above right] {\small $v_0 = v_3'$};
  \node at ($(W1) + (-0.1,0,0.1)$) [below left] {\small $v_1$};
  \node at ($(W2) + (0.1,-0.1,-0.1)$) [right] {\small $v_2$};
  \node at ($(W3) + (0,-0.05,0.05)$) [right] {\small $v_3 = v_0'$};

  \node at ($(U1) + (0,-0.33,-0.05)$) [right] {\small $v_1'$};
  \node at ($(U2) + (-0.05,-0.03,0)$) [left] {\small $v_2'$};

\end{tikzpicture}
\end{adjustbox}

\caption{$F_i$ and $F_i'$ for $i \neq 0,n$.}
\label{figure_mutation2}
\end{figure}

\noindent
Suppose that $F_i$ is admissible. Then it is enough to prove that 
\[
	v_0' - v_j' \equiv 0 \quad (\mathrm{mod} ~d_i')
\]
for every $j=1,\dots,\hat{i}, \dots, n-1$. We first observe that 
\[
	F_i' = \mathrm{conv}\{v_0', \dots, \hat{v}_i', \dots, v_n'\} = 
	\mathrm{conv} \left\{ v_n, \left\{v_n - \frac{h_{\max}}{h_{\min}}(v_j-v_0) ~|~ j \in \{1,\dots,n-1\} \setminus \{i\} \right\}, v_0 
	\right\}
\]
which implies that $F_i$ and $F_i'$ lie in the same hyperplane. In particular, they have the same affine distance from the origin and so we have 
\[
	d_i = d_i'.
\] 
Moreover, the admissibility of $F_i$ implies that
\begin{itemize}
	\item $v_0' - v_n' = v_n - v_0 \equiv 0 ~(\mathrm{mod}~d_i)$ since $v_0, v_n \in F_i$, and
	\item $v_0' - v_j' = \displaystyle \underbrace{\frac{1}{d_n}(v_0 - v_j)}_{\text{integral vector}} \times ~h_{\max}$ for $j \in \{1,\dots,n-1\} \setminus \{i\}$.
\end{itemize}
We also record the following divisibility properties. Let
\[
j\in \{1,\ldots,n-1\}\setminus\{i\}.
\]

\begin{itemize}
\item Since both $F_i$ and $F_n$ are admissible and
$
v_0,v_j\in F_i\cap F_n,
$
we have
\[
d_i \mid (v_0-v_j)
\qquad\text{and}\qquad
d_n \mid (v_0-v_j), 
\]
and therefore
$
\operatorname{lcm}(d_i,d_n)\mid (v_0-v_j).
$
\vs{0.1cm}

\item Since
\[
h_{\max}
=
\langle -n_{F_n},v_n\rangle
=
\langle -n_{F_n},(v_n-v_0) + v_0 \rangle
=
\langle -n_{F_n},v_n-v_0\rangle-d_n,
\]
we have
\[
h_{\max}+d_n
=
\langle -n_{F_n},v_n-v_0\rangle .
\]
On the other hand, the admissibility of $F_i$ implies
\[
d_i\mid (v_n-v_0).
\]
Thus
\[
d_i
\mid
\langle -n_{F_n},v_n-v_0\rangle
=
h_{\max}+d_n.
\]
Consequently,
\[
\gcd(d_i,d_n)\mid h_{\max}.
\]
\end{itemize}

\noindent
Thus we have $d_id_n = \mathrm{lcm}(d_i,d_n) \cdot \mathrm{gcd}(d_i,d_n) ~|~ (v_0 - v_j) \cdot h_{\max}$,
and therefore
\[
	d_i ~|~ \displaystyle \frac{1}{d_n}(v_0 - v_j) \times h_{\max} = v_0' - v_j'.
\]
This proves our claim. 
See Figure \ref{figure_mutation2}; the blue and red triangles indicate $F_1$ and $F_1'$, respectively.

\vs{0.1cm}
\noindent
{\bf Case III: $i=0$.} 
\vs{0.1cm}

\tdplotsetmaincoords{70}{120}

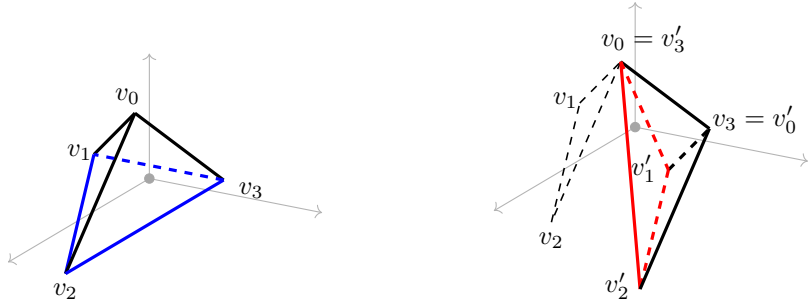
\begin{figure}[htbp]
\centering
\begin{adjustbox}{center}

\begin{tikzpicture}[tdplot_main_coords, scale=2.2]

  \draw[gray!60,->] (0,0,0) -- (1.7,0,0) node[below right] {\small $$};
  \draw[gray!60,->] (0,0,0) -- (0,1.2,0) node[above left] {\small $$};
  \draw[gray!60,->] (0,0,0) -- (0,0,0.8) node[above] {\small $$};
  \filldraw[gray!70] (0,0,0) circle (0.8pt);

  \coordinate (W2) at (1.7, 0.4, 0);
  \coordinate (W3) at (-0.2, 0.4, 0);
  \coordinate (W0) at (0, -0.1, 0.4);
  \coordinate (W1) at (-0.2, -0.5, 0);

  \draw[very thick, blue] (W2) -- (W1);
  \draw[very thick] (W1) -- (W0);
  \draw[very thick] (W0) -- (W2);

  \draw[very thick, blue] (W2) -- (W3);
  \draw[very thick] (W3) -- (W0);
  \draw[dashed, very thick, blue] (W3) -- (W1);

  \node at ($(W0) + (0.1,-0.15,0)$) [above right] {\small $v_0$};
  \node at ($(W1) + (-0.1,0,0.1)$) [below left] {\small $v_1$};
  \node at ($(W2) + (0.1,-0.1,-0.1)$) [right] {\small $v_2$};
  \node at ($(W3) + (0,0.35,0)$) [left] {\small $v_3$};

\end{tikzpicture}
\hspace{2cm}

\begin{tikzpicture}[tdplot_main_coords, scale=2.2]

  \draw[gray!60,->] (0,0,0) -- (1.7,0,0) node[below right] {\small $$};
  \draw[gray!60,->] (0,0,0) -- (0,1.2,0) node[above left] {\small $$};
  \draw[gray!60,->] (0,0,0) -- (0,0,0.8) node[above] {\small $$};
  \filldraw[gray!70] (0,0,0) circle (0.8pt);

  \coordinate (W2) at (1.7, 0.4, 0);
  \coordinate (W3) at (-0.2, 0.4, 0);
  \coordinate (W0) at (0, -0.1, 0.4);
  \coordinate (W1) at (-0.2, -0.5, 0);

  \coordinate (U1) at (-0.4, -0.0, -0.4);
  \coordinate (U2) at (1.5, 0.9, -0.4);

  \draw[dashed, semithick] (W2) -- (W1);
  \draw[dashed, semithick] (W1) -- (W0);
  \draw[dashed, semithick] (W0) -- (W2);

  \draw[very thick] (W3) -- (W0);

  \draw[very thick, red] (W0) -- (U2);
  \draw[dashed, very thick] (U1) -- (W3);
  \draw[dashed, very thick, red] (U1) -- (U2);
  \draw[very thick] (U2) -- (W3);

  \draw[dashed, very thick, red] (W0) -- (U1);

  \node at ($(W0) + (0.1,-0.15,0)$) [above right] {\small $v_0 = v_3'$};
  \node at ($(W1) + (-0.1,0,0.1)$) [below left] {\small $v_1$};
  \node at ($(W2) + (0.1,-0.1,-0.1)$) [right] {\small $v_2$};
  \node at ($(W3) + (0,-0.05,0.05)$) [right] {\small $v_3 = v_0'$};

  \node at ($(U1) + (0,-0.33,-0.05)$) [right] {\small $v_1'$};
  \node at ($(U2) + (-0.05,-0.03,0)$) [left] {\small $v_2'$};

\end{tikzpicture}
\end{adjustbox}

\caption{$F_0$ and $F_0'$}
\label{figure_mutation3}
\end{figure}

\noindent
Assume that $F_0$ is admissible. To prove that $F_0'$ is also admissible, we need to show that
\[
	v_n' - v_j' \equiv 0 \quad (\mathrm{mod} ~d_0')
\]
for every $j=1,\dots,n-1$. Consider the decomposition
\begin{equation}\label{eq_decomposition}
	v_n' - v_j' = \underbrace{(v_0' - v_j')}_{\text{(1)}} + \underbrace{(v_j - v_0')}_{\text{(2)}}  + \underbrace{(v_n' - v_j)}_{\text{(3)}}
\end{equation}
and observe that 
\[
	(1) = \text{$v_0' - v_j' = \displaystyle \frac{h_{\max}}{d_n(=-h_{\min})}(v_0 - v_j) = \displaystyle \frac{\langle {\bf w}_n, v_n \rangle}{d_n}(v_0 - v_j)$.}
\]
Since $v_n' = v_0$, we have 
\[
	\text{(1) + (3)} = \frac{d_n + \langle {\bf w}_n, v_n \rangle}{d_n}(v_0-v_j).
\]
Since $d_n + \langle {\bf w}_n, v_n \rangle = {\bf w}_n \cdot (v_n - v_j)$ and 
$(v_n - v_j)$ is divisible by $d_0$, the scalar
\[
d_n+\langle {\bf w}_n,v_n\rangle
\]
is divisible by \(d_0\). Moreover, since $F_n$ is admissible, the vector 
\[
	\frac{1}{d_n}(v_0-v_j)
\]
is integral. Thus the entire vector (1) + (3) is divisible by $d_0$.
For the vector (2), note that $v_j - v_0' = v_j - v_n$ is divisible by $d_0$ by the admissibility of $F_0$. Therefore, we obtain
\[
	v_n' - v_j' = (1) + (2) + (3) \equiv 0 \quad (\mathrm{mod} ~d_0) \quad \text{for every $j=1,\dots,n-1$.}
\]

We complete the proof by showing that
\[
d_0' \mid d_0.
\]
More precisely, we will prove that $F_0'$ is contained in a hyperplane of the form
\begin{equation}\label{eq_new_hyperplane}
{\bf u}_0' \cdot {\bf x} = -d_0,
\end{equation}
where ${\bf u}_0'$ is an integral (not necessarily primitive) vector. Since $d_0'$ is the lattice distance from the origin to $F_0'$, it then follows that
\[
d_0' \mid d_0.
\]
To do this, let us consider the integral vector
\[
	{\bf u}_0' := d_n {\bf w}_0 + (d_0 + \langle {\bf w}_0, v_0 \rangle) {\bf w}_n. 
\]
Clearly, ${\bf u}_0'$ is a nonzero vector (since ${\bf w}_0$ and ${\bf w}_n$ are linearly independent). Moreover,
\[
	\begin{array}{ccl}\vs{0.2cm}
		{\bf u}_0' \cdot \underbrace{(v_n' - v_j')}_{=(1)+(2)+(3)} & = & {\blue{\langle d_n {\bf w}_0, v_n' - v_j' \rangle}} + {\red{\langle (d_0 + \langle {\bf w}_0, v_0 \rangle) {\bf w}_n, v_n' - v_j' \rangle}} \\ \vs{0.1cm} 
					& = & \displaystyle {\blue{\langle d_n {\bf w}_0, \underbrace{\frac{d_n + \langle {\bf w}_n, v_n \rangle}{d_n}(v_0-v_j)}_{=(1)+(3)} \rangle}}
+ {\red{\langle (d_0 + \langle {\bf w}_0, v_0 \rangle) {\bf w}_n, \underbrace{v_j - v_n}_{=(2)} \rangle }}
\\ \vs{0.2cm}
					& = & \langle {\bf w}_n, v_n - v_j \rangle \left( 
			\langle {\bf w}_0, v_0 - v_j \rangle - (d_0 + \langle {\bf w}_0, v_0 \rangle)
\right)
\quad (\text{since} ~d_n = -\langle {\bf w}_n, v_j \rangle)
\\
					& = & 0 \quad (\text{since} ~\langle {\bf w}_0, v_j \rangle = -d_0)
	\end{array}
\]
for every $j = 1,\dots,n-1$. Here, we used the fact that the term (2) in \eqref{eq_decomposition} is orthogonal to ${\bf w}_0$ and the term (1) + (3) 
is orthogonal to ${\bf w}_n$. 
Thus ${\bf u}_0'$ is normal to $F_0'$ and 
the equation of the hyperplane containing $F_0'$ is given by 
\begin{equation}\label{eq_hyperplane}
	{\bf u}_0' \cdot {\bf x} = {\bf u}_0' \cdot v_n' = {\bf u}_0' \cdot v_0 = -d_0d_n \quad \text{(since ${\bf w}_n \cdot v_0 = -d_n$.)}
\end{equation}
On the other hand, we note that 
\[
	d_0 + \langle {\bf w}_0, v_0 \rangle = \langle {\bf w}_0, v_0 - v_j \rangle \quad \quad j \neq 0, n,
\]
and $v_0 - v_j$ is divisible by $d_n$ by the admissibility of $F_n$. Therefore, we have 
$
	d_n ~|~ d_0 + \langle {\bf w}_0, v_0 \rangle
$
and the hyperplane equation \eqref{eq_hyperplane} can be reduced into
\[
	\frac{1}{d_n} {\bf u}_0' \cdot {\bf x} = -d_0
\]
where $\displaystyle \frac{1}{d_n} {\bf u}_0'$ is integral. This implies that the affine distance $d_0'$ of the hyperplane should divide $d_0$. This completes the proof.
See Figure \ref{figure_mutation3}; the blue and red triangles indicate $F_0$ and $F_0'$, respectively.

\end{proof}

\begin{remark}
	In the proof of Theorem \ref{thm_main}, we saw that 
	\[
		d_n' = h_{\max},  \quad d_i' = d_i ~(i \neq 0, n),
		\quad \text{and that \hs{0.1cm} $d_0' ~|~ d_0$.}
	\]
	Applying the same argument to the inverse mutation $\mu_{F_n'}$, 
	one immediately checks that $d_0 ~|~ d_0'$. Hence we obtain $d_0' = d_0$. Summing up, we get
	\begin{equation}\label{eq_aff_distance_change}
		d_i' = \begin{cases}
			d_i & i = 0, \dots, n-1 \\
			h_{\max} & i=n, \quad \text{$h_{\max} = \langle {\bf w}_n, v_n \rangle$: maximal height of
			$P$ with respect to ${\bf w}_n$.}
		\end{cases}
	\end{equation}  
\end{remark}

\begin{example} Let $P$ be a Fano triangle with vertices $(0,1), (1,-2),$ and $(-1,-2)$. By calculating supporting hyperplanes of $P$, one can easily check that $\mathrm{ad}(P) = 3$. The mutation $\mu_F(P)$ along the facet $F = \mathrm{conv} \{ (-1,-2), (1,-2) \}$ 
has three edges with defining equations
\[
	y - 3x = 1, \quad y = 1, \quad 3x - 2y = 1,
\]
and hence we have $\mathrm{ad}(\mu_F(P)) = 3$. Similarly, it follows from the definition that any $n$-dimensional reflexive simplex has $\mathrm{ad} = n+1$. See Figure \ref{figure_mutation_Fano_triangle_full_rank}.

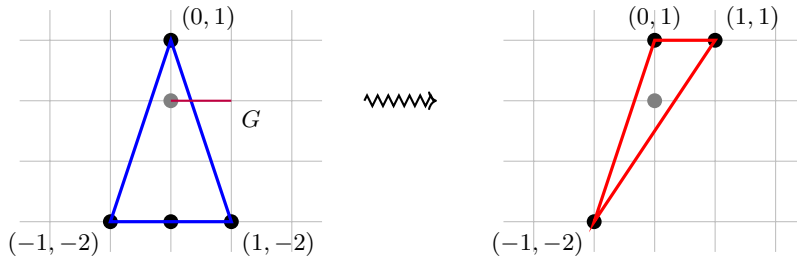
\begin{figure}[htbp]
\centering
\begin{center}
\begin{tikzpicture}[scale=0.8]

  \begin{scope}
    \draw[step=1cm, gray!50, thin] (-2.5,-2.5) grid (2.5,1.5);

    \fill[gray] (0,0) circle (3.5pt);
    \fill[black] (0,1) circle (3.5pt);
    \fill[black] (1,-2) circle (3.5pt);
    \fill[black] (-1,-2) circle (3.5pt);
    \fill[black] (0,-2) circle (3.5pt);
    \draw[very thick, blue] (1,-2) -- (0,1) -- (-1,-2) -- cycle;

 \draw[thick, purple] (0,0) -- (1,0);

    \node[below right] at (1,0) {\footnotesize $G$};
    \node[above right] at (0,1) {\footnotesize $(0,1)$};
    \node[below right] at (1,-2) {\footnotesize $(1,-2)$};
    \node[below left] at (-1,-2) {\footnotesize $(-1,-2)$};
  \end{scope}

  \draw[decorate, decoration={zigzag, segment length=4, amplitude=2}, thick, ->]
    (3.2,0) -- (4.4,0);

  \begin{scope}[xshift=8.0cm]
    \draw[step=1cm, gray!50, thin] (-2.5,-2.5) grid (2.5,1.5);
    \fill[gray] (0,0) circle (3.5pt);
    \fill[black] (0,1) circle (3.5pt);
    \fill[black] (-1,-2) circle (3.5pt);
    \fill[black] (1,1) circle (3.5pt);
    \draw[very thick, red] (0,1) -- (-1,-2) -- (1,1) -- cycle;
    \node[above] at (0,1) {\footnotesize $(0,1)$};
    \node[below left] at (-1,-2) {\footnotesize $(-1,-2)$};
    \node[above right] at (1,1) {\footnotesize $(1,1)$};
  \end{scope}

\end{tikzpicture}
\end{center}
\caption{Fano triangle with $\mathrm{ad}(P) = 3$ and its combinatorial mutation}
\label{figure_mutation_Fano_triangle_full_rank}
\end{figure}

\end{example}

\begin{example} Let $P$ be the $n$-dimensional standard Fano simplex defined by 
	\[
		P = \mathrm{conv} \{{\bf e}_0 := -\sum_{i=1}^n {\bf e}_i, {\bf e}_1, \dots, {\bf e}_n\}
	\]
	where ${\bf e}_i$ is the $i$-th unit coordinate vector for $i=1,\dots,n$.

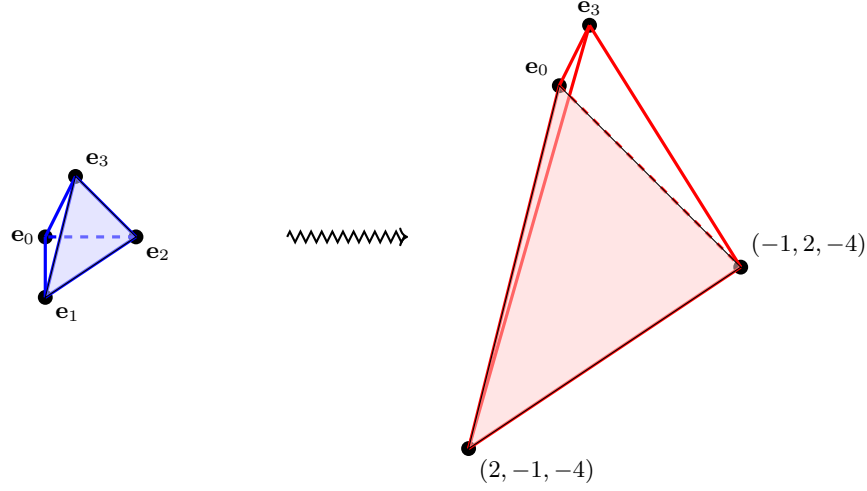
\begin{figure}[htbp]
\centering
\begin{center}
\begin{tikzpicture}[scale=0.8]

  \begin{scope}[yshift=-2.5cm]

    \fill[black] (0,1) circle (3.5pt);
    \fill[black] (1,0) circle (3.5pt);
    \fill[black] (-0.5,-1) circle (3.5pt);
    \fill[black] (-0.5,0) circle (3.5pt);
	\draw[very thick, blue] (0,1) -- (1,0);
	\draw[very thick, blue] (-0.5,-1) -- (1,0);
	\draw[very thick, blue] (-0.5,-1) -- (-0.5,0);
	\draw[very thick, blue] (0,1) -- (-0.5,0);
	\draw[very thick, blue] (0,1) -- (-0.5,-1);
	\draw[dashed, very thick, blue] (1,0) -- (-0.5,0);


    \node[below right] at (1,0) {\footnotesize ${\bf e}_2$};
    \node[above right] at (0,1) {\footnotesize ${\bf e}_3$};
    \node[below right] at (-0.5,-1) {\footnotesize ${\bf e}_1$};
    \node[left] at (-0.5,0) {\footnotesize ${\bf e}_0$};

	\draw[thin, fill=blue!20, fill opacity=0.5] (0,1) -- (1,0) -- (-0.5,-1) -- cycle;
  \end{scope}

	\begin{scope}[yshift=-2.5cm]
  \draw[decorate, decoration={zigzag, segment length=4, amplitude=2}, thick, ->]
    (3.5,0) -- (5.5,0);
	\end{scope}

  \begin{scope}[xshift=8.5cm]

    \fill[black] (0,1) circle (3.5pt);
    \fill[black] (-0.5,0) circle (3.5pt);
    \fill[black] (2.5,-3) circle (3.5pt);
    \fill[black] (-2,-6) circle (3.5pt);

    \draw[very thick, red] (0,1) -- (-0.5,0);
    \draw[very thick, red] (0,1) -- (2.5,-3);
    \draw[very thick, red] (0,1) -- (-2,-6);
    \draw[very thick, red] (-0.5,0) -- (-2,-6);
    \draw[dashed, very thick, red] (-0.5,0) -- (2.5,-3);
    \draw[very thick, red] (2.5,-3) -- (-2,-6);

    \node[above right] at (2.5,-3) {\footnotesize $(-1,2,-4)$};
    \node[below right] at (-2,-6) {\footnotesize $(2,-1,-4)$};
    \node[above left] at (-0.5,0) {\footnotesize ${\bf e}_0$};
    \node[above] at (0,1) {\footnotesize ${\bf e}_3$};

	\draw[thin, fill=red!20, fill opacity=0.5] (-0.5,0) -- (2.5,-3) -- (-2,-6) -- cycle;


  \end{scope}

\end{tikzpicture}
\end{center}
\caption{Mutation of the standard 3-simplex along $\mathrm{conv}\{{\bf e}_1, {\bf e}_2, {\bf e}_3\}$}
\label{figure_mutation_3simplex}
\end{figure}

\noindent
Since $P$ is reflexive, every facet is admissible and so $\mathrm{ad}(P)=n+1$. If we take a factor $F = \mathrm{conv}\{{\bf e}_1, \dots, {\bf e}_n \}$ with the height vector ${\bf w} = (-1, \dots, -1)$ and a vertex $v = {\bf e}_n$, the combinatorial mutation of $P$ along $F$ is given by 
\[
	\mu_F(P) = \mathrm{conv}\{{\bf e}_n, {\bf e}_0 + h_{\max}G_{F,v}\}, \quad \quad 
	G_{F,v} = \frac{1}{d_F} (F - v) = F - v
\]
where $h_{\max} = \langle {\bf w}, {\bf e}_0 \rangle = n$ and $d_F = 1$ as the hyperplane equation of $F$ is $x_1 + \dots + x_n = 1$. Then the mutated simplex $\mu_F(P)$ has $n+1$ hyperplanes whose equations are 
\[
	x_1 + \dots + x_n = -n, \quad x_1 + \dots + x_n - (n+1)x_i = 1 \hs{0.2cm} (i \in [n-1]),
	\quad (n+2)(x_1 + \dots + x_{n-1}) + x_n = 1.
\]
Note that the primitive inward normal vector ${\bf u}_0'$ of $F_0'$ in \eqref{eq_new_hyperplane} is exactly 
\[
	{\bf u}_0' = 	d_n {\bf w}_0 + (d_0 + \langle {\bf w}_0, v_0 \rangle) {\bf w}_n
	= (-(n+2),\dots,-(n+2),-1)
\]
since ${\bf w}_0 = (-1,\dots,-1,n)$, ${\bf w}_n = (-1,\dots, -1)$, $d_0 = d_n = 1$, 
and $\langle {\bf w}_0, v_0 \rangle = n$.

\end{example}

\begin{remark}
	In \cite{AK2}, Akhtar and Kasprzyk proved a result analogous to Theorem \ref{thm_main}
	in the two-dimensional case, namely
	that combinatorial mutation preserves the {\em singularity content} of a Fano polygon.
	For a given Fano polygon $P$ with the spanning fan $\Sigma$ and maximal (simplicial) cones 
	$\{\sigma_i\} \subset \Sigma$, let $u_i$ and $v_i$ be the primitive ray vectors generating 
	$\sigma_i$,
	that is, the vertices of $P$ connected by an edge. We denote by $w_i$ the lattice length of 
	$\mathrm{conv}\{u_i, v_i\}$ and $\ell_i$ the affine distance 
	of $\mathrm{conv}\{u_i, v_i\}$ from the origin. Using the elementary division algorithm, we have 
	\[
		w_i = n_i \ell_i + \rho_i, \quad 0 \leq \rho_i < \ell_i.
	\]
	A {\em singularity content} of a Fano polytope $P$ consists of a pair $(n, \mcal{B})$ where 
	$n = \sum n_i$ with $w_i = n_i \ell_i + \rho_i$ 
	and $\mcal{B}$ is the collection of {\em residual data}. (See \cite{AK2} for more detail.)
	Then $\mathrm{ad}(P) = k$ if and only if exactly $k$ of the $\rho_i$ are equal to zero. Since the singularity content of $P$ 
	is preserved under combinatorial mutation by \cite{AK2}, so is $\mathrm{ad}(P)$. In geometric terms,
	the number $\mathrm{ad}(P)$ is the number of $T$-singularities in the toric Fano surface $X_P$ associated
	to $P$. See \cite[Corollary 2.6]{AK2}. 
\end{remark}

\section{Markov-type equations of Fano simplices}
\label{secMarkovtypeEquationsOfFanoSimplices}

In this section, we establish a bridge between the geometry of Fano simplices and the arithmetic of Diophantine equations. To each Fano simplex we associate a weighted Markov-type equation together with a distinguished positive integer solution. We then show that combinatorial mutations over admissible facets induce Vieta-type arithmetic mutations on the corresponding solutions. As a consequence, the geometric mutation dynamics of Fano simplices is reflected in the arithmetic mutation dynamics of the associated Diophantine data.
%

Consider the equation of the form
\begin{equation}\label{eq_GDE}
		c_0x_0^2 + c_1x_1^2 + \dots + c_nx_n^2 = qx_0x_1 \dots x_n
\end{equation}
for some real numbers $c_0, c_1, \dots, c_n$, and $q$. 
Analogously to the classical Vieta involutions for the Markov equation, positive solutions of the above weighted Markov-type equation generate new solutions through the following Vieta-type mutation.
The proof is a straightforward computation and is therefore omitted.

\begin{lemma}\label{lemma_Vieta}
Suppose that 
${\bf a}:= (a_0, a_1, \dots, a_n) \in (\R_{>0})^{n+1}$ is a solution of the equation \eqref{eq_GDE}. 
For each $i=0, \dots, n$, there exists a positive solution 
$\mu_i({\bf a}) = (a_0', a_1', \dots, a_n') \in (\R_{>0})^{n+1}$ where 
\[
	a_j' = 
	\begin{cases}
		a_j & \text{if}~ j \neq i \\
		\displaystyle \frac{c_0a_0^2 + \dots + \widehat{c_i a_i^2} + \dots + c_na_n^2}{c_ia_i} & \text{if} ~j=i	
	\end{cases}
\]
Furthermore,  $\mu_i$ is an involution, i.e., $\mu_i^2({\bf a}) = {\bf a}$. 
\end{lemma}


Motivated by the work of Akhtar-Kasprzyk \cite{AK1}, we will assign a Diophantine equation for each Fano simplex as follows. Let $P \subset N_\R$ be a Fano simplex of dimension $n$ with facets $F_0, \dots, F_n$ and $\Sigma_P$ the spanning fan of $P$. Denote by $\sigma_i$ the maximal cone in $\Sigma_P$ associated to $F_i$ and let $\widetilde{\lambda}_i$ be the index of the sub-lattice of $N$ spanned by the primitive ray generators of $\sigma_i$. Also we define $d := \mathrm{gcd}(\widetilde{\lambda}_0, \dots, \widetilde{\lambda}_n)$ called the {\em multiplicity} of $P$. Equivalently, $d$ equals to the cardinality of the quotient group $N/N'$ where $N'$ is the sublattice of $N$ generated by the vertices of $P$.
We also denote by 
\[
	\displaystyle \lambda:= (\lambda_0, \dots, \lambda_n) := \frac{1}{d}(\widetilde{\lambda}_0, \dots, 
	\widetilde{\lambda}_n)
\]	
called the {\em weight} of $P$. It is straightforward by definition that a weight is integral and primitive. 
Moreover, it follows directly from the definition of a Fano simplex that the associated weights are {\em well-formed} (cf.~\cite[p.~3, first paragraph]{CGKN}); that is,
\[
\gcd(\lambda_0,\dots,\widehat{\lambda_i},\dots,\lambda_n)=1
\quad \text{for all } i=0,\dots,n.
\]

Now we consider the factorization $\widetilde{\lambda}_i = dc_ia_i^2$ where $c_i, a_i \in \Z_{>0}$ and $c_i$ is square-free. 
Then it is obvious that ${\bf a}_P := (a_0, a_1, \dots, a_n)$ is a positive integral solution of the Diophantine equation
\begin{equation}\label{eq_DE_for_Fano}
	k(c_0x_0^2 + c_1x_1^2 + \dots + c_nx_n^2) = mx_0x_1\dots x_n
\end{equation}
where $k$ and $m$ are positive integers such that 
\[
	\displaystyle \frac{m}{k} = \frac{c_0a_0^2 + c_1a_1^2 + \dots + c_na_n^2}{a_0 a_1\dots a_n} \in \Q, \quad (k,m) = 1. 
\]
We call the equation \eqref{eq_DE_for_Fano}
{\em the weighted Markov-type equation for $P$} and denote it by $\mathrm{DE}_P$.



\begin{proposition}[{\cite[Theorem 12 and Remark 14]{CGKN}}] \label{prop_CGKN}
Let $P \subset N_\R$ be a Fano simplex with weight
\[
	\lambda=(\lambda_0,\dots,\lambda_n)\in (\Z_{>0})^{n+1}
\]
and multiplicity $d>0$. Let $Q$ be a combinatorial mutation of $P$ over the facet $F_i$. Then the weight $\lambda' \in (\Z_{>0})^{n+1}$ and the multiplicity $d' > 0$ are given by 
\[
	\lambda'=(\lambda_0,\dots,\lambda_i',\dots,\lambda_n), \quad 
	\lambda_i' = \frac{\bigl(\lambda_0+\cdots+\widehat{\lambda_i}+\cdots+\lambda_n\bigr)^2} {\lambda_i}
\]
and 
\[
	d' = d \cdot
\left(
\frac{\lambda_0+\cdots+\widehat{\lambda_i}+\cdots+\lambda_n}
{\lambda_i}
\right)^{n-2}.
\]
\end{proposition}

Using Proposition \ref{prop_CGKN}, we can prove the following. 

\begin{theorem}\label{thm_DE_inv}
	Let $P$ be a Fano simplex.
	If $Q$ is a combinatorial mutation of $P$ over an admissible facet, then 
	\[
			\mathrm{DE}_P = \mathrm{DE}_Q.
	\]
	In particular, $\mathrm{DE}_P$ depends only on the facet-mutation class of P.
\end{theorem}

\begin{proof}
	
	Let $\mathrm{DE}_P$ be the Markov-type equation for $P$ given by 
	\[
		k(c_0x_0^2 + c_1x_1^2 + \dots + c_nx_n^2) = mx_0x_1\dots x_n
	\]
	where $\lambda_i = c_ia_i^2$ with $c_i, a_i \in \Z_{>0}$ and $c_i$ is square-free for every $i=0, \dots, n$. By Proposition \ref{prop_CGKN}, the weights of $Q$ are given by 
   	\[
   		\lambda_j' = \begin{cases}
   			\lambda_j & \text{if $j \neq i$,} \\
   			\displaystyle \frac{(\lambda_0 + \dots + \hat{\lambda}_i + \dots + \lambda_n)^2}{\lambda_i} & \text{if $j = i$}.
   		\end{cases}
   	\]
	By direct calculation, we have 
	\[
		\begin{array}{ccccl}\vs{0.2cm}
			\lambda_i' & = &  \displaystyle \frac{(c_0a_0^2 + \dots + \hat{c_ia_i^2} + \dots + c_na_n^2)^2}{c_ia_i^2} 
			& = & c_i \displaystyle \left(\frac{c_0a_0^2 + \dots + \hat{c_ia_i^2} + \dots + c_na_n^2}{c_ia_i} \right)^2.
			\\
		\end{array}
	\]
	Therefore it suffices to show that 
	$(a_1, \dots, a_{i-1}, a_i', a_{i+1}, \dots, a_n)$ is an integral solution to $\mathrm{DE}_P$ where 
	\[
		a_i' := \frac{c_0a_0^2 + \dots + \hat{c_ia_i^2} + \dots + c_na_n^2}{c_ia_i} \in \Q
	\]	
	and is equal to ${\bf a}_Q$. 
	Lemma \ref{lemma_Vieta} implies that ${\mu}_i({\bf a}) = (a_1, \dots, a_{i-1}, a_i', a_{i+1}, \dots, a_n)$ and hence is a solution 
	to $\mathrm{DE}_P$. Also, the integrality of $a_i'$ follows from the fact that 
	$\lambda_i'$ is integral and $c_i$ is square-free. 
	Thus we have ${\bf a}_Q = (a_1, \dots, a_{i-1}, a_i', a_{i+1}, \dots, a_n)$ and this 
	completes the proof.
	 
\end{proof}

Theorem~\ref{thm_DE_inv} establishes a natural assignment
\[
\{ \text{Fano simplices} \} / \sim \;\xrightarrow{\;\phi\;}\; \{ \text{Markov-type equations} \}
\]
where $``\sim"$ denotes unimodular equivalence. 
More precisely, to each Fano polytope $P$, one can associate a weighted Markov-type equation $\mathrm{DE}_P$ together with an integer solution ${\bf a}_P \in \Z^{n+1}$. We emphasize that this correspondence is not bijective. For instance, consider a Fano triangle $P$ associated with the fake weighted projective plane $\p(5,5,5)$ with multiplicity $d=5$. In this case, $\mathrm{DE}_P$ coincides with the Markov equation
\begin{equation}\label{eq_Markov}
x^2 + y^2 + z^2 = 3xyz,
\end{equation}
which is also the equation corresponding to $\p(1,1,1) = \C P^2$. Indeed, one can easily check that two Fano simplices with the same weights (but possibly different multiplicities) give rise to the same Diophantine equation.
On the other hand, the correspondence between polytopes and solutions behaves differently. There are infinitely many positive integer solutions (called Markov triples) to \eqref{eq_Markov} as we have seen in Section \ref{secIntroduction}.
In fact, there is a one-to-one correspondence between the set of such solutions and the set of mutation equivalence classes of Fano triangles associated with $\p(1,1,1)$.
In contrast, there are only two facet-mutation equivalence classes of Fano triangles associated with $P$, namely those corresponding to $\p(5,5,5)$ and $\p(5,5,20)$.
See Figure \ref{figure_mutation_triangle_example}.
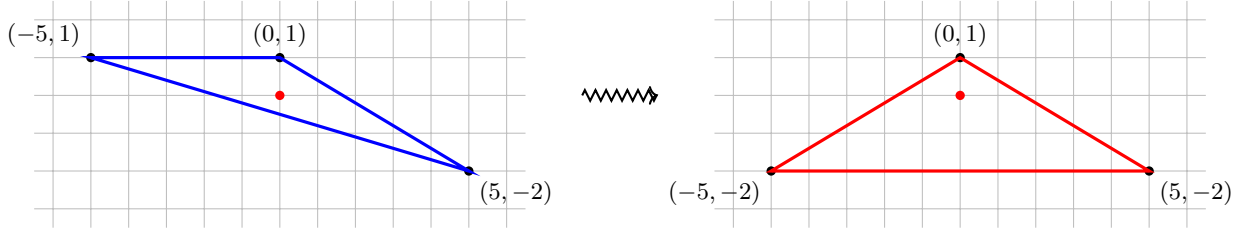
\begin{figure}[htbp]
\centering
\begin{tikzpicture}[scale=0.5]

  \begin{scope}
    \draw[step=1cm, gray!50, thin] (-6.5,-3.5) grid (6.5,2.5);

    \fill[black] (-5,1) circle (3.5pt);
    \fill[black] (0,1) circle (3.5pt);
    \fill[black] (5,-2) circle (3.5pt);
    \fill[red] (0,0) circle (3.5pt);

    \draw[very thick, blue] (-5,1) -- (0,1) -- (5,-2) -- cycle;

    \node[above left] at (-5,1) {\footnotesize $(-5,1)$};
    \node[above] at (0,1) {\footnotesize $(0,1)$};
    \node[below right] at (5,-2) {\footnotesize $(5,-2)$};
  \end{scope}

  \draw[decorate, decoration={zigzag, segment length=4, amplitude=2}, thick, ->]
    (8,0) -- (10,0);

  \begin{scope}[xshift=18cm]
    \draw[step=1cm, gray!50, thin] (-6.5,-3.5) grid (6.5,2.5);

    \fill[black] (0,1) circle (3.5pt);
    \fill[black] (-5,-2) circle (3.5pt);
    \fill[black] (5,-2) circle (3.5pt);
    \fill[red] (0,0) circle (3.5pt);

    \draw[very thick, red] (0,1) -- (-5,-2) -- (5,-2) -- cycle;

    \node[above] at (0,1) {\footnotesize $(0,1)$};
    \node[below left] at (-5,-2) {\footnotesize $(-5,-2)$};
    \node[below right] at (5,-2) {\footnotesize $(5,-2)$};
  \end{scope}

\end{tikzpicture}
\caption{A Fano triangle for $\p(5,5,5)$ and its combinatorial mutation over a facet}
\label{figure_mutation_triangle_example}
\end{figure}

%

\begin{corollary}
	Let $P$ be a reflexive simplex and $P^\vee$ the polar dual of $P$. Then 
	\[
		\mathrm{DE}_P = \mathrm{DE}_{P^\vee}.
	\]
\end{corollary}

\begin{proof}
	It was proved by Conrad \cite[Lemma 5.3]{Con} that a reflexive simplex $P$ and its polar
	dual $P^\vee$ have the same weights, but possibly different multiplicities. 
	Thus, the result follows immediately.
\end{proof}

\begin{example}
	In dimension two, it is known that there are 5 reflexive simplices whose weights and the multiplicities are listed as 
	\[
		\ell_1 = \{ (1,1,1), d = 1 \}, \quad 
		\ell_2 = \{ (1,1,1), d = 3 \}, \quad
		\ell_3 = \{ (1,1,2), d = 1 \}, \quad
		\ell_4 = \{ (1,1,2), d = 2 \}, 
	\]
	and $\ell_5 = \{ (1,2,3), d = 1 \}$. 
	Here, $\ell_1$ is dual to $\ell_2$, $\ell_3$ is dual to $\ell_4$, and $\ell_5$ is self-dual.
	One can see that the associated Markov-type equation is given by 
	\[
		\ell_1, \ell_2: x^2 + y^2 + z^2 = 3xyz \quad \quad
	 	\ell_3, \ell_4: x^2 + y^2 + 2z^2 = 4xyz \quad \quad 
		\ell_5: x^2 + 2y^2 + 3z^2 = 6xyz.
	\]
\end{example}

\begin{remark}
	The phenomenon described in Theorem~\ref{thm_DE_inv} is specific to admissible-facet mutations.
	In contrast, for general combinatorial mutations of Fano simplices,
	one does not expect the existence of a mutation-invariant
	Diophantine equation. For instance $\p(1,1,1,3)$ and $\p(1,1,4,6)$ are mutation equivalent while the corresponding Markov-type equations 
	are $x^2 + y^2 + z^2 + 3w^2 = 6xyzw$ and $x^2 + y^2 + z^2 + 6w^2 = 6xyzw$, respectively.
	Indeed, Coates-Gonshaw-Kasprzyk-Nabijou
	observed that several arithmetic quantities associated with the
	weights fail to be preserved under mutation in higher dimensions \cite[Remark~17]{CGKN}.
	Theorem~\ref{thm_DE_inv} shows that admissible-facet mutations form
	a distinguished subclass for which a weighted Markov-type equation
	is preserved.
\end{remark}

\section{Dual simplices and Sliding operators}
\label{secDualSimplicesAndSlidingOperators}

We have seen in Section~\ref{secMarkovtypeEquationsOfFanoSimplices} that each Fano simplex $P$ gives rise to a Markov-type Diophantine equation $\mathrm{DE}_P$ together with a distinguished positive integer solution ${\bf a}_P \in \Z^{n+1}$. Moreover, the associated arithmetic mutations reflect certain structural features of the admissible-facet mutation class of $P$.
In this section, we develop a dual-polytope interpretation of combinatorial mutation. More precisely, we introduce an invertible piecewise linear transformation on $M_\R$, called a \emph{sliding operator}, and show that combinatorial mutation over an admissible facet is realized by this transformation in the dual picture. As a consequence, the sliding operator induces a bijection between the dual polytopes $P^\vee$ and $Q^\vee$, and in particular preserves their volumes. 
In contrast to approaches in which this formula is interpreted via anticanonical degrees or related toric intersection-theoretic considerations ~\cite{Ba,Ka,Ni}, our treatment is entirely polyhedral. This viewpoint provides a direct geometric explanation of volume preservation under mutation and yields an alternative derivation of the multiplicity change formula in Proposition~\ref{prop_CGKN}. We also obtain a volume formula for $P^\vee$ in terms of the coefficients of $\mathrm{DE}_P$ and the multiplicity of $P$.

We begin by defining a sliding operator on $M_\R$. Informally, a sliding operator moves each line segment in $\Delta$ parallel to ${\bf w}$ toward the facet $D$ while preserving its length. The following definition makes this precise; see also Figure~\ref{figure_sliding}.

\begin{definition}\label{def_sliding}
	For a given convex polytope $\Delta \subset M_\R$ containing the origin in its interior, let ${\bf w} \in \mathrm{vert}(\Delta)$ and $D$ a facet of $\Delta$ containing ${\bf w}$. 
	Denote by $H_D$ the hyperplane of $M_\R$ containing $D$ and $H_D^+$ the half-space supported by $H_D$ and containing $\Delta$. 
	A {\em sliding of $\Delta$ along $({\bf w},D)$} is the convex polytope $\mu_{{\bf w},D}(\Delta) \subset H_D^+$ such that 
	\begin{itemize}
		\item for any line $\ell$ in direction ${\bf w}$, it satisfies
		\[
			|\ell \cap \Delta| = |\ell \cap \mu_{{\bf w},D}(\Delta)|
		\]
		where $|\cdot|$ denotes the length, and  
		\item $\ell \cap \mu_{{\bf w},D}(\Delta) \cap H_D$ consists of a single point whenever $\ell\cap\Delta\neq\emptyset$.
	\end{itemize}
\end{definition}

In fact, this definition is not entirely new. Pabiniak and Tolman \cite[Definition 2.5]{PT} introduced an operation that transforms a lattice polytope into another lattice polytope in order to study toric degenerations of smooth projective varieties. 
We will not describe it in detail here, but in fact Pabiniak--Tolman’s operation coincides with our sliding operator in a specific situation, namely when $\Delta$ and $\mu_{{\bf w}, D}(\Delta)$ are both integral simplices. This situation rarely arises in our setting since the dual of a Fano simplex is not integral in general. 

We now show that the sliding operator introduced above realizes combinatorial mutation in the dual picture. More precisely, we relate it to the piecewise linear transformation appearing in the work of Akhtar et al.~\cite{ACGK}, who described combinatorial mutations in terms of dual polytopes.

\begin{proposition}\cite[Proof of Proposition 4]{ACGK}\label{prop_ACGK}
	Let $P \subset N_\R$ be a Fano polytope and let $Q = \mathrm{mut}_{\bf w}(P,G;\{L_h\})$ be the combinatorial mutation along the data $({\bf w}, G, \{L_h\})$ as in Definition 
	\ref{def_comb_mutation}. Then the map 
	\[
		\begin{array}{ccccl}\vs{0.1cm}
			\varphi & : & M_\R & \rightarrow & M_\R \\
					  &   &  m & \mapsto & m - m_{\min}{\bf w} \\
		\end{array}	
	\]
	where $m_{\min} := \mathrm{min} \{\langle m, v \rangle ~|~ v \in \mathrm{vert}(G) \}$ maps $P^\vee$ bijectively onto $Q^\vee$.
\end{proposition}

Let $P$ be a Fano simplex and $F$ an admissible facet of $P$. Choose a vertex $v_0 \in F$ and let $G$ be a lattice polytope obtained by translating $F$ by $-v_0$ and scaling by $\frac{1}{d_F}$ so that 
\[
	G := G_{F,v_0} = \frac{1}{d_F} \left(F - v_0 \right).
\]
Let ${\bf w}_F \in M_\R$ be the primitive inward normal vector of $P$ to $F$ and define 
\[
	Q := \mu_{F,v_0}(P) = \mathrm{mut}_{{\bf w}_F}(P,G;\{L_h\})
\]
where the collection $\{L_h\}_{h\in\mathbb Z_{<0}}$
is given by
\[
L_{-d_F}=\{v_0\},
\]
and $L_h=\emptyset$ for all $h\neq -d_F$.
Then the piecewise linear map in Proposition \ref{prop_ACGK} is given by 
\[
	\varphi:= \varphi_{F,v_0} : M_{\mathbb{R}} \to M_{\mathbb{R}}, \quad m \mapsto m - m_{\min}\, {\bf w}_F.
\]
The map $\varphi$ consists of $|G|$ linear pieces where $|G|$ denote the number of vertices of $G$.
Accordingly, we label each linear piece by the corresponding vertex of $G$.

\begin{lemma}\label{lemma_SLQ}
	Each linear piece lies in $\mathrm{SL}_n(\mathbb{Z})$. Consequently, $P^\vee$ and $Q^\vee$ have the same volume, the same number of
	lattice points, and the same number of interior lattice points.
\end{lemma}

\begin{proof}
Each linear piece corresponding to a vertex $v \in \mathrm{Vert}(G)$ is given by 
\[
	\varphi_{v} := m \mapsto m - \langle m, v \rangle {\bf w}_F, 
\]
which coincides with the restriction of $\varphi$ to the subdomain 
\[
	\{ m \in M_\R ~|~ m_{\min} = \langle m, v \rangle \}.
\] 
Since $\langle G, {\bf w}_F \rangle = 0$, we also have $\langle v, {\bf w}_F \rangle = 0$ for every $v \in \mathrm{Vert}(G)$.

Take any integral basis $\mathcal{B} = \{ e_1 := {\bf w}_F, e_2, \dots, e_n \} \subseteq M$ of $M_\R$. Then the matrix representation of 
$\varphi_v$ with respect to $\mathcal{B}$ is of the form
\[
[\varphi_v]_{\mathcal{B}} =
\begin{pmatrix}
1 & a_{12} & \cdots & a_{1n} \\
0 & 1 &        & 0 \\
\vdots &  & \ddots & \\
0 & 0 &        & 1
\end{pmatrix},
\quad \quad a_{1j} = - \langle e_j, v \rangle \in \Z
\]
which lies in $\mathrm{SL}_n(\mathbb{Z})$. This proves the lemma.
\end{proof}

Let $p \in P^\vee$ and let $\ell_p$ denote the maximal line segment in $P^\vee$ passing through $p$ and parallel to ${\bf w}_F$, namely
\[
	\ell_p := \{\, p+t{\bf w}_F \mid t \in \R \,\} \cap P^\vee.
\]

\begin{lemma}\label{lemma_linear_piece}
	Let $p \in P^\vee$ and let
	\[
		\varphi_v : m \mapsto m-\langle m,v\rangle{\bf w}_F
	\]
	be the linear piece whose domain contains $p$, for some $v\in\mathrm{Vert}(G)$.
	Then every point on $\ell_p$ belongs to the domain of $\varphi_v$. Moreover, the domain of $\varphi_v$ is precisely the union of ${\bf w}_F$-segments starting from $D_v$, that is, 
	\[
		\left\{ p+t{\bf w}_F \in P^\vee ~\middle|~ p\in P^\vee,\; 0\le t\le t_p,\; p+t_p{\bf w}_F\in D_v \right\}
	\]
	where $D_v$ denotes the facet of $P^\vee$ dual to the vertex $v_0+d_Fv$ in $F$.
\end{lemma}	
\begin{proof}
        Since $\langle {\bf w}_F, v' \rangle = 0$ for every $v' \in G$, we have
	\[
		(p+t{\bf w}_F)_{\min} = \min_{v' \in \mathrm{Vert}(G)} \langle p + t{\bf w}_F, v' \rangle = \min_{v' \in \mathrm{Vert}(G)} \langle p , v' \rangle = \langle p , v \rangle,
	\]
	which proves the first claim.
	
	Recall that a point $m \in P^\vee$ is in the domain of $\varphi_v$ if 
	\[
		m_{\min} = \min_{v' \in \mathrm{Vert}(G)} \langle m, v' \rangle = \langle m, v \rangle.
	\]
	It is equivalent to saying that 
	\[
		m_{\min} = \min_{v' \in \mathrm{Vert}(G)} \langle m, v_0 + d_Fv' \rangle = \min_{u' \in \mathrm{Vert}(F)} \langle m, u' \rangle = \langle m, u \rangle
	\]	
	where $u = v_0 + d_F v$ is the vertex of $P$ corresponding to $v$. 
	Now let $f \in D_v$. Since 
	\[
		\langle f, v_0 + d_Fv \rangle = -1,
	\]
	which is the minimum of
	\[
		\{\langle f,u'\rangle \mid u'\in \mathrm{Vert}(F)\},
	\]
	every point on the maximal line segment through $f$ parallel to ${\bf w}_F$ belongs to the domain of $\varphi_v$ by the first claim. 
	This proves the second claim.
\end{proof}

\begin{proposition}\label{prop_sliding}
	Let $D(v_0)$ be the facet of $P^\vee$ dual to $v_0 \in \mathrm{Vert}(P)$. Then the sliding of $P^\vee$ along $({\bf w}_F, D(v_0))$
	coincides with $Q^\vee$. In other words, 
	\[
		\mu_{{\bf w}_F, D(v_0)}(P^\vee) = Q^\vee.
	\]
\end{proposition}

\begin{proof}
	Let $p\in P^\vee$ and let
	\[
		\varphi_v: m \mapsto m-\langle m,v\rangle{\bf w}_F
	\]
	be the linear piece whose domain contains $p$. By Lemma~\ref{lemma_linear_piece}, every point of $\ell_p$ belongs to the domain of $\varphi_v$.

	We first show that $\varphi$ preserves the length of each segment $\ell_p$.	
	Since $\langle {\bf w}_F,v\rangle=0$, the restriction of $\varphi_v$ to $\ell_p$ is given by
	\[
		\varphi_v ~|_{\ell_p} : p + t{\bf w}_F \mapsto p + t{\bf w}_F - \langle p + t{\bf w}_F, v \rangle {\bf w}_F = p + t{\bf w}_F - \langle p, v \rangle {\bf w}_F
	\]
	Hence $\varphi_v$ acts on $\ell_p$ by translation along the vector $-\langle p,v\rangle{\bf w}_F$. In particular, we have 
	\[
		|\ell_p|=|\varphi_v(\ell_p)|.
	\]
	
	Next, we show that $\varphi_v(\ell_p)$ meets $H_{D(v_0)}$ in exactly one point. Let $p+t_p{\bf w}_F$ be the endpoint of $\ell_p$ lying on the facet $D_v$.
	By Lemma \ref{lemma_linear_piece}, 
	\[
		\langle p+t_p{\bf w}_F, v_0+d_Fv \rangle =-1.
	\]
	Using $\langle {\bf w}_F, v_0 \rangle = -d_F$, we obtain
	\[
		\begin{aligned}
			\left\langle p+t_p{\bf w}_F - \langle p+t_p{\bf w}_F,v\rangle {\bf w}_F, \, v_0 \right\rangle &= \left\langle p+t_p{\bf w}_F, \, v_0-\langle {\bf w}_F,v_0\rangle v \right\rangle \\ &
			= \left\langle p+t_p{\bf w}_F, \, v_0 + d_F v \right\rangle \\ &
			= -1.
		\end{aligned}
	\]
	Hence
	$
		\varphi_v(p+t_p{\bf w}_F)\in H_{D(v_0)}.
	$
	Since $\varphi_v$ acts on $\ell_p$ by translation, it follows that
	\[
		\varphi_v(\ell_p)\cap H_{D(v_0)}
	\]
	consists of a single point. This completes the proof.
\end{proof}

\begin{example} Let $P = \text{conv}({\bf e}_1,{\bf e}_2, -{\bf e}_1 - {\bf e}_2)$ be the Fano simplex of $\p(1,1,1)$ and $Q$ be the mutation of $P$ along $({\bf w}, G)$ given in Example \ref{example_standard_simplex}.
Then the sliding of $P^\vee$ onto $Q^\vee$ can be described as in Figure \ref{figure_sliding}.

\begin{figure}[H]
\centering
\begin{tikzpicture}[x=0.8cm,y=0.8cm]


\begin{scope}

\draw[->] (-2.5,0) -- (3.5,0) node[right] {$x$};
\draw[->] (0,-2) -- (0,3) node[above] {$y$};

\filldraw[fill=blue,fill opacity=0.25,draw=black,thick]
(-1,-1)--(-1,2)--(2,-1)--cycle;

\draw[blue,very thick] (-1,-1)--(2,-1);

\fill (-1,-1) circle (2pt);
\fill (-1,2) circle (2pt);
\fill (2,-1) circle (2pt);

\node[left] at (-1.3,2) {$(-1,2)$};
\node[below left] at (-1,-1) {$(-1,-1)$};
\node[below right] at (2,-1) {$(2,-1)$};

\node[blue] at (0.5,-1.35) {$D$};

\foreach \c in {-1,0,1,2}
{
\draw[dashed,gray]
(-2,-2+\c)--(3,3+\c);
}

\draw[->,thick]
(2.2,2.2)--(1.3,1.3);

\node[right] at (2.2,2.2)
{$\mathbf w=(-1,-1)$};

\node at (-1.3,3)
{$\Delta = P^\vee$};

\end{scope}


\draw[->,very thick]
(4.5,0.5)--(6.5,0.5);

\node at (5.5,0.9)
{sliding};


\begin{scope}[xshift=10cm]

\draw[->] (-5,0) -- (3,0) node[right] {$x$};
\draw[->] (0,-2) -- (0,3) node[above] {$y$};

\filldraw[fill=red,fill opacity=0.25,draw=black,thick]
(-4,-1)--(.5,.5)--(2,-1)--cycle;

\draw[blue,very thick] (-4,-1)--(2,-1);

\fill (-4,-1) circle (2pt);
\fill (.5,.5) circle (2pt);
\fill (2,-1) circle (2pt);

\node[below left] at (-4,-1)
{$(-4,-1)$};

\node[above] at (.5,.5)
{$\left(\frac12,\frac12\right)$};

\node[below right] at (2,-1)
{$(2,-1)$};

\node[blue] at (-1.5,-1.35) {$D$};

\foreach \c in {-1,0,1,2}
{
\draw[dashed,gray]
(-2,-2+\c)--(3,3+\c);
}

\filldraw[fill=blue,fill opacity=0.01,draw=black,thin]
(-1,-1)--(-1,2)--(2,-1)--cycle;

\node at (-2,3)
{$\mu_{\mathbf w,D}(\Delta) = Q^\vee$};

\end{scope}

\end{tikzpicture}
\caption{Sliding of $\Delta$ along $(\mathbf w,D)$.}
\label{figure_sliding}
\end{figure}
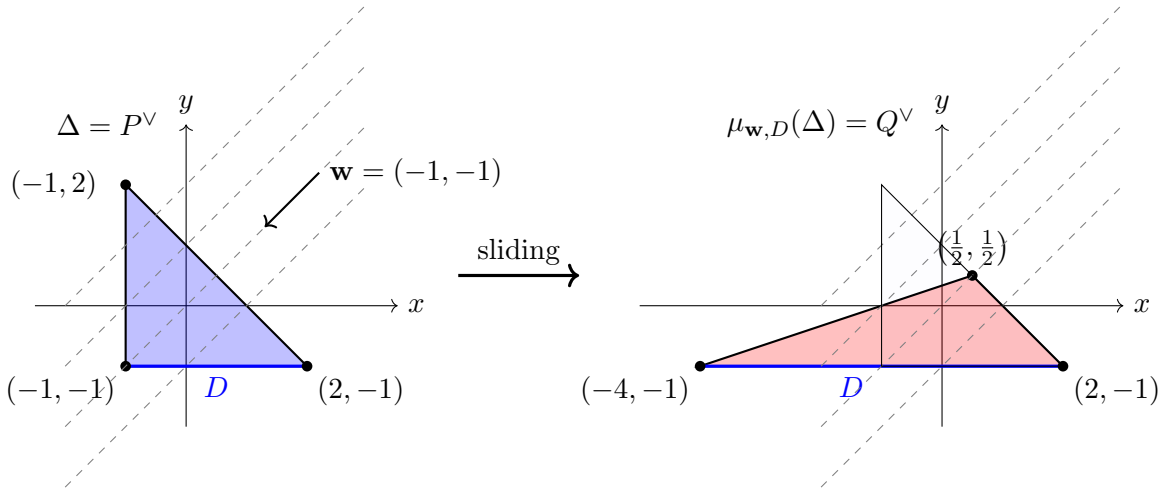

\end{example}

Our final goal of this section is to calculate an explicit formula for the volume of $P^\vee$ in terms of the coefficients of $\mathrm{DE}_P$ and the multiplicity of $P$. 
This formula also yields an alternative proof of the multiplicity change formula in Proposition~\ref{prop_CGKN}.

\begin{theorem}\label{thm_multiplicity}
Let $P$ be a Fano $n$-simplex corresponding to a solution $(a_0,\dots,a_n)$ of the weighted Markov-type equation $\mathrm{DE}_P$ given by 
\[
c_0x_0^2+\cdots+c_nx_n^2 = m x_0\cdots x_n.
\]
Then
\[
\Vol(P^\vee)
=
\frac{m^n}{n!\,\prod c_i}
\cdot
\frac{(\prod a_i)^{n-2}}{\mult(P)}.
\]
In particular, if $P$ and $Q$ are related by a mutation over an admissible facet and ${\bf a}_Q = (a_0', \dots, a_n')$, then 
\[
\frac{(\prod a_i)^{n-2}}{\mult(P)}
=
\frac{(\prod a_i')^{n-2}}{\mult(Q)}.
\]
\end{theorem}

\begin{remark}

We note that the quantity
\[
\frac{(\prod_i a_i)^{n-2}}{\mult(P)}
\]
is closely related to the anticanonical degree of the fake weighted projective space associated to a Fano simplex. Indeed, the volume formula
\[
\Vol(P^\vee)
=
\frac{\left(\sum_i \lambda_i\right)^n}
{n!\,\mult(P)\prod_i \lambda_i}
\]
appears naturally in several contexts, including reflexive simplices~\cite{Ni} and fake weighted projective spaces~\cite{Ka,Ba}. Our derivation arises from the dual-polytope interpretation of combinatorial mutation via sliding operators. In particular, this viewpoint yields a geometric explanation of the multiplicity change formula in Theorem~\ref{thm_multiplicity}.
\end{remark}

To prove Theorem \ref{thm_multiplicity}, we setup the following. 
Let
$
P=\operatorname{conv}(v_0,\dots,v_n)\subset N_{\mathbb R}
$
be an $n$-dimensional Fano simplex with the primitive positive relation among all vertices
\[
	\lambda_0v_0+\cdots+\lambda_nv_n=0, \quad \lambda_i\in \mathbb Z_{>0} \quad\text{for all }i, \qquad \gcd(\lambda_0,\dots,\lambda_n)=1.
\]
Let $N':=\mathbb Z v_0+\cdots+\mathbb Z v_n\subset N$ be a sublattice of $N$ generated by the vertices of $P$ and denote by 
\[
	\mult(P):=[N:N']
\]
the multiplicity of $P$. 
For each $i=0,\dots,n$, let
	\[
		F_i:=\operatorname{conv}(v_0,\dots,\widehat{v_i},\dots,v_n)
	\]
be the facet opposite to $v_i$, and let
\[
N_i:=\mathbb Z v_0+\cdots+\widehat{\mathbb Z v_i}+\cdots+\mathbb Z v_n
\subset N'.
\]
Note that each $N_i$ has rank $n$.

\begin{lemma}\label{lemma_index}
	For each $i=0,\dots,n$, one has
	\[
		[N':N_i]=\lambda_i.
	\]
\end{lemma}

\begin{proof}
	Indeed, in the quotient group $N'/N_i$, the classes of $v_j$ for $j\neq i$ vanish, so
$N'/N_i$ is generated by the class of $v_i$.
From the primitive relation
\[
\lambda_0v_0+\cdots+\lambda_nv_n=0,
\]
we obtain in $N'/N_i$ the relation
\[
\lambda_i[v_i]=0.
\]
Thus the order of $[v_i]$ divides $\lambda_i$.

To see that the order is exactly $\lambda_i$, suppose that
\[
k[v_i]=0
\qquad\text{in }N'/N_i
\]
for some positive integer $k$. Then
\[
kv_i\in N_i,
\]
so there exist integers $a_j$ ($j\neq i$) such that
\[
kv_i+\sum_{j\neq i} a_jv_j=0.
\]
Since the relation
\[
\lambda_0v_0+\cdots+\lambda_nv_n=0
\]
is primitive and unique up to sign among the vertices of the simplex, the vector
\[
(a_0,\dots,a_{i-1},k,a_{i+1},\dots,a_n)
\]
must be an integral multiple of
\[
(\lambda_0,\dots,\lambda_n).
\]
In particular, $\lambda_i$ divides $k$.
Hence the order of $[v_i]$ is exactly $\lambda_i$, and therefore
\[
[N':N_i]=|N'/N_i|=\lambda_i.
\]
This proves the claim.

\end{proof}

\begin{proposition}\label{prop:dual-volume-weight-multiplicity}
Then the volume of the polar dual simplex $P^\vee\subset M_{\mathbb R}$ is
\[
\Vol(P^\vee)
=
\frac{(\lambda_0+\cdots+\lambda_n)^n}
{n!\,\mult(P)\,\lambda_0\cdots\lambda_n}.
\]
\end{proposition}

\begin{proof}
We first define
	\[
		\delta_i:=[N:N_i] = [N:N']\,[N':N_i]. \qquad (i=0,\dots,n).
	\]
Then, by Lemma \ref{lemma_index}, we have 
\[
	\delta_i=\mult(P)\,\lambda_i. 
\]
In particular, we have 
\begin{equation}\label{eq_relation}
	\delta_0v_0+\cdots+\delta_nv_n = \mult(P)\bigl(\lambda_0v_0+\cdots+\lambda_nv_n\bigr)=0.
\end{equation}

To compute the volume of $P^\vee$, we first recall that 
\[
	P^\vee:=\{u\in M_{\mathbb R}\mid \langle u,v\rangle\ge -1 \text{ for all }v\in P\}.
\]
For each $i=0,\dots,n$, let $u_i\in M_{\mathbb R}$ be the vertex of $P^\vee$
corresponding to the facet $F_i$.
Then by definition,
\[
\langle u_i,v_j\rangle=-1
\qquad\text{for all }j\neq i.
\]
Set
\[
h:=\delta_0+\cdots+\delta_n.
\]
Since $\delta_i=\mult(P)\lambda_i$, we also have
\[
h=\mult(P)(\lambda_0+\cdots+\lambda_n).
\]
Pairing the relation \eqref{eq_relation} with $u_i$, we get
\[
0
=
\left\langle u_i,\sum_{j=0}^n \delta_jv_j\right\rangle
=
\delta_i\langle u_i,v_i\rangle+\sum_{j\neq i}\delta_j\langle u_i,v_j\rangle.
\]
Since $\langle u_i,v_j\rangle=-1$ for $j\neq i$, this becomes
\[
0
=
\delta_i\langle u_i,v_i\rangle-\sum_{j\neq i}\delta_j
=
\delta_i\langle u_i,v_i\rangle-(h-\delta_i).
\]
Therefore
\[
\langle u_i,v_i\rangle=\frac{h-\delta_i}{\delta_i}
=\frac{h}{\delta_i}-1.
\]

Fix $u_0$ as a base vertex of $P^\vee$.
For $i=1,\dots,n$ and $k=1,\dots,n$, we compute
\[
\langle u_i-u_0,v_k\rangle.
\]
If $k\neq i$, then
\[
\langle u_i,v_k\rangle=-1
\qquad\text{and}\qquad
\langle u_0,v_k\rangle=-1,
\]
hence
\[
\langle u_i-u_0,v_k\rangle=0.
\]
If $k=i$, then
\[
\langle u_i-u_0,v_i\rangle
=
\langle u_i,v_i\rangle-\langle u_0,v_i\rangle
=
\left(\frac{h}{\delta_i}-1\right)-(-1)
=
\frac{h}{\delta_i}.
\]
Thus we have 
\[
\langle u_i-u_0,v_k\rangle
=
\begin{cases}
\dfrac{h}{\delta_i}, & k=i,\\[4pt]
0, & k\neq i.
\end{cases}
\]

Let $\beta_1,\dots,\beta_n\in M_{\mathbb Q}$ be the $\mathbb Q$-dual basis to
$v_1,\dots,v_n$, namely
\[
\langle \beta_i,v_j\rangle=\delta_{ij}.
\]
Then the above pairing identities imply
\[
u_i-u_0=\frac{h}{\delta_i}\beta_i
\qquad (i=1,\dots,n).
\]

Taking determinants, we obtain
\[
\det(u_1-u_0,\dots,u_n-u_0)
=
\frac{h^n}{\delta_1\cdots\delta_n}\det(\beta_1,\dots,\beta_n).
\]
Since $\beta_1,\dots,\beta_n$ is dual to $v_1,\dots,v_n$,
\[
\det(\beta_1,\dots,\beta_n)=\det(v_1,\dots,v_n)^{-1}.
\]
Taking absolute values gives
\[
\left|\det(u_1-u_0,\dots,u_n-u_0)\right|
=
\frac{h^n}{\delta_1\cdots\delta_n}\cdot
\frac{1}{|\det(v_1,\dots,v_n)|}.
\]
But the sublattice generated by $v_1,\dots,v_n$ is exactly $N_0$. Therefore, we have 
\[
	|\det(v_1,\dots,v_n)|=[N:N_0] = \delta_0.
\]
Hence
\[
\left|\det(u_1-u_0,\dots,u_n-u_0)\right|
=
\frac{h^n}{\delta_0\delta_1\cdots\delta_n}.
\]
Consequently, 
\[
\Vol(P^\vee)
=
\frac{1}{n!}\left|\det(u_1-u_0,\dots,u_n-u_0)\right|
=
\frac{h^n}{n!\,\delta_0\cdots\delta_n}.
\]

Finally, substituting
\[
h=\mult(P)(\lambda_0+\cdots+\lambda_n)
\qquad\text{and}\qquad
\delta_i=\mult(P)\lambda_i
\]
for all $i$, we obtain
\[
\Vol(P^\vee)
=
\frac{\bigl(\mult(P)(\lambda_0+\cdots+\lambda_n)\bigr)^n}
{n!\,\prod_{i=0}^n \bigl(\mult(P)\lambda_i\bigr)}.
=
\frac{(\lambda_0+\cdots+\lambda_n)^n}
{n!\,\mult(P)\,\lambda_0\cdots\lambda_n}
\]
This proves the proposition.
\end{proof}

\begin{proof}[Proof of Theorem \ref{thm_multiplicity}]
	By assumption, we have 
		\[
			\lambda_0+\cdots+\lambda_n = c_0a_0^2+\cdots+c_na_n^2 = m a_0\cdots a_n,
		\]
	and
		\[
			\lambda_0\cdots\lambda_n = (c_0\cdots c_n)(a_0\cdots a_n)^2.
		\]
	Substituting these identities into
	\[
		\Vol(P^\vee) = \frac{(\lambda_0+\cdots+\lambda_n)^n} {n!\,\mult(P)\,\lambda_0\cdots\lambda_n}
	\]
	gives
	\[
		\Vol(P^\vee) = \frac{(m a_0\cdots a_n)^n} {n!\,\mult(P)\,(c_0\cdots c_n)(a_0\cdots a_n)^2} = \frac{m^n (a_0\cdots a_n)^{n-2}} {n!\,\mult(P)\,c_0\cdots c_n}.
	\]
\end{proof}

%

%
%
%

\newcommand{\etalchar}[1]{$^{#1}$}
\providecommand{\bysame}{\leavevmode\hbox to3em{\hrulefill}\thinspace}
\providecommand{\MR}{\relax\ifhmode\unskip\space\fi MR }
\providecommand{\MRhref}[2]{%
	\href{http://www.ams.org/mathscinet-getitem?mr=#1}{#2}
}
\providecommand{\href}[2]{#2}

\end{document}